\documentclass[10pt]{article}
\usepackage[utf8]{inputenc}
\usepackage[english]{babel}
\usepackage{xcolor}

\usepackage{enumerate}

\usepackage{algpseudocode}
\usepackage{algorithm}
\usepackage{array}

\usepackage{graphicx}
\usepackage{caption}
\usepackage{subcaption}

\usepackage{amsfonts}
\usepackage{amsthm,amsmath,amssymb}

\usepackage{comment}

\usepackage{authblk}

\usepackage{geometry}
\usepackage{amsthm}
\newtheorem{theorem}{Theorem}[section]
\newtheorem{lemma}[theorem]{Lemma}
\newtheorem{proposition}[theorem]{Proposition}
\newtheorem{corollary}[theorem]{Corollary}
\newtheorem{problem}{Problem}[section]
\newtheorem{remark}[theorem]{Remark}
\newtheorem{definition}[theorem]{Definition}
\newtheorem{assumption}{Assumption}[section]

\usepackage{datetime}

\date{\displaydate{date}}

\numberwithin{equation}{section}

\usepackage{scalerel,stackengine}
\stackMath
\newcommand\reallywidehat[1]{%
\savestack{\tmpbox}{\stretchto{%
  \scaleto{%
    \scalerel*[\widthof{\ensuremath{#1}}]{\kern.1pt\mathchar"0362\kern.1pt}%
    {\rule{0ex}{\textheight}}
  }{\textheight}%
}{2.4ex}}%
\stackon[-6.9pt]{#1}{\tmpbox}%
}

\begin{document}

\title{\LARGE \bf A phase-field approach for detecting cavities via a Kohn-Vogelius type functional}

\author[1]{Andrea Aspri}
\affil[1]{Department of Mathematics, Università degli Studi di Milano}

\vspace{5mm}

\date{}

\maketitle

\thispagestyle{plain}
\pagestyle{plain}

\let\thefootnote\relax\footnotetext{
AMS 2020 subject classifications: 35R30, 65N21, 74G75

Key words and phrases: Kohn-Vogelius functional, cavity, phase-field, linear elasticity, primal dual active set method

\thanks{}

}
\begin{abstract}
We deal with the geometrical inverse problem of the shape reconstruction of cavities in a bounded linear isotropic medium by means of boundary data. The problem is addressed from the point of view of optimal control: the goal is to minimize in the class of Lipschitz domains a Kohn-Vogelius type functional with a perimeter regularization term which penalizes the perimeter of the cavity to be reconstructed. To solve numerically the optimization problem, we use a phase-field approach, approximating the perimeter functional with a Modica-Mortola relaxation and modeling the cavity as an inclusion with a very small elastic tensor. We provide a detailed analysis showing the robustness of the algorithm through some numerical experiments.     
\end{abstract}

\maketitle


\section{Introduction}\label{sec:introduction}
The main focus of this paper is to propose an efficient and robust algorithm, based on a phase-field approach, to address the geometrical inverse problem of identification of cavities contained in an elastic isotropic body, utilizing tractions and displacement boundary measurements. We work in the framework of linear elasticity, representing the medium by a bounded domain $\Omega\subset \mathbb{R}^d$, where $d=2,3$. 
These kinds of inverse problems appear in non-destructive testing techniques used by industry to detect defects, voids, cracks in a medium, which can appear during manufacturing processes, and to evaluate properties of materials and structures without causing damage to the medium (\cite{AmmBreEliGarJossKan15,BouMulSahTan21,DizLauBen16,JavHol19,Lang19}). For instance, non-destructive testing methods are particularly important in the fields where the new techniques of additive manufacturing are replacing traditional methods of metal manufacture \cite{EilUrb18,KurMarIya17,NgoKasImbNguHui18,TroWelElv18}. \\
Let $\Omega$ be a bounded domain, with $\partial\Omega:=\Sigma_N\cup\Sigma_D$, where $\Sigma_D$ is closed. Given a bounded Lipschitz domain $C$, with $C\subset \Omega$, we consider the following boundary value problem
\begin{equation}\label{eq:neumann_intro}
    \begin{cases}
    \textrm{div}(\mathbb{C}_0\widehat{\nabla}u_N)=0 & \textrm{in}\ \Omega\setminus C\\
    (\mathbb{C}_0\widehat{\nabla}u_N)n=0 & \textrm{on}\ \partial C\\
    (\mathbb{C}_0\widehat{\nabla}u_N)\nu=g & \textrm{on}\ \Sigma_N\\
    u_N=0 &\textrm{on}\ \Sigma_D,
    \end{cases}
\end{equation}
where $n, \nu$ are the unit outer normal vectors to $C, \Sigma_N$, respectively, $\mathbb{C}_0$ is a fourth order elastic tensor, uniformly bounded, strongly convex and satisfying minor and major symmetries, and $\widehat{\nabla} u_N$ represents the deformation tensor. We assume that $g\in L^2(\Sigma_N)$.

Given $\mathbb{C}_0$, $C$, and $g$, the forward or direct problem corresponds to find the solution $u_N$ in $\Omega$.
On the contrary, the inverse problem consists in the identification of the cavity $C$ given $\mathbb{C}_0$, and $g$, and making use of the additional boundary measurements represented by the displacement vector $f=u_N\lfloor_{\Sigma_N}$. It has been proved that uniqueness for cavities detection holds in the class of Lipschitz domains \cite{MorRos04,ABR18} while stability estimates (of logarithmic type) have been proved for more regular cavities, precisely assuming a-priori $C^{1,\alpha}$ regularity, with $0<\alpha\leq 1$ (\cite{MorRos04}). Similar stability estimates hold also in the case of elastic inclusions (\cite{MR16}).\\
Due to the very weak stability estimates, identification of cavities (and also inclusions) from boundary measurements is an ill-posed problem which needs of regularization techniques to be solved. In \cite{Rondi11,RinRon11} a phase-field method has been applied for the reconstruction of cracks and cavities in the case of the conductivity equation.
Recently, a phase-field approach has been proposed in \cite{DecEllSty16} and then applied also in \cite{BerRatVer18} for the identification of inclusions in the framework of a linear and a semilinear elliptic equation, respectively. The same approach has been also extended to the detection of cavities in the case of a semilinear elliptic equation in \cite{BCP2021} and of linear elasticity in \cite{AspBerCavRocVer22}. All these papers propose an algorithm rephrasing the inverse problem as an optimization procedure, where the goal is to minimize a suitable misfit functional, defined on the boundary of $\Omega$, with the addition of a regularization term which involves a relaxation of the perimeter of the domain to be reconstructed. 

A similar point of view is utilized in this paper, i.e., we apply again a phase-field method but this time for the minimization of a Kohn-Vogelius type functional (\cite{KohVog87}), that is an energy-gap functional, regularized with a penalization on the perimeter of the cavity. To the author's knowledge, phase field methods have never been applied, in the inverse problems context, to Kohn-Vogelius type functionals. More precisely, we consider the minimization of the following functional
\begin{equation*}
     J_{reg}(C)=J_{KV}(C)+\alpha\ \textrm{Per}(C),
\end{equation*}
where $\textrm{Per}(C)$ is the perimeter of the set $C$, $\alpha$ is the so-called regularization parameter, and $J_{KV}$ is a Kohn-Vogelius type functional defined as
\begin{equation*}
J_{KV}(C)=\frac{1}{2}\int_{\Omega\setminus C}\mathbb{C}_0\widehat{\nabla}{(u_N(C)-u_D(C))}:\widehat{\nabla}{(u_N(C)-u_D(C))}\, dx.
\end{equation*}
The states $u_N(C)$ and $u_D(C)$ are, respectively, solutions to the problem \eqref{eq:neumann_intro} and
\begin{equation*}
     \begin{cases}
    \textrm{div}(\mathbb{C}_0\widehat{\nabla}u_D)=0 & \textrm{in}\ \Omega\setminus C\\
    (\mathbb{C}_0\widehat{\nabla}u_D)n=0 & \textrm{on}\ \partial C\\
    u_D=f & \textrm{on}\ \Sigma_N\\
    u_D=0 &\textrm{on}\ \Sigma_D.
    \end{cases} 
\end{equation*}
The first part of the paper is devoted to prove the existence of minima for the functional $J_{reg}$. This result follows by showing the continuity of the functional $J_{KV}$ with respect to perturbations of $C$ in the Hausdorff metric which is obtained by means of the Mosco convergence \cite{BucBut05,BHSZ01,Gia04,HenPie18}.   

The second part of the paper concerns with the phase-field relaxation of the functional $J_{reg}$ in order to obtain a continuous and Frech\'et differentiable functional on a convex subset of $H^1(\Omega)$. More precisely, we adopt the same strategy applied in the optimization field (see, for example, \cite{BouCha03}): assuming $\mathbb{C}_0$ extended in the whole domain $\Omega$, we fill the cavity with a fictitious material with a very small elastic tensor, that is we define $\mathbb{C}_1=\delta \mathbb{C}_0$, where $\delta>0$ is a small parameter. Introducing a phase field variable $v$ which belongs to $H^1$ and takes values in the interval $[0,1]$, and using the Modica-Mortola relaxation of the perimeter, see \cite{Mod87}, we study the following functional
\begin{equation*}
J_{\delta,\varepsilon}(v):= J^{\delta}_{KV}(v)
+ \gamma \!\int_{\Omega}\Big( \varepsilon|\nabla v|^2 + \frac{1}{\varepsilon}v(1-v)\Big)\, dx,
\end{equation*}
where $\gamma$ is a suitable rescaling parameter, and $J^{\delta}_{KV}(v)$ is defined as
\begin{equation*}
J^{\delta}_{KV}(v)=\frac{1}{2}\int_{\Omega}\mathbb{C}^{\delta}(v)\widehat{\nabla}{(u^{\delta}_N(v)-u^{\delta}_D(v))}:\widehat{\nabla}{(u^{\delta}_N(v)-u^{\delta}_D(v))}\, dx,
\end{equation*}
where $\mathbb{C}^{\delta}(v)= \mathbb{C}_0 +  (\mathbb{C}_1- \mathbb{C}_0)v$, and the states $u^{\delta}_N(v)$ and $u^{\delta}_D(v)$ are solutions to 
\begin{equation*}
    \begin{cases}
    	\textrm{div}(\mathbb{C}^{\delta}(v) \widehat{\nabla} u^{\delta}_N(v)) &= 0 \qquad \text{in}\ \Omega,	\\
		(\mathbb{C}^{\delta}(v) \widehat\nabla u^{\delta}_N(v)) \nu &= g \qquad \text{on}\ \Sigma_N,\\
		u^{\delta}_N(v)&=0 \qquad \text{on}\ \Sigma_D,
    \end{cases}
\quad \textrm{and}\quad 
\begin{cases}
    	\textrm{div}(\mathbb{C}^{\delta}(v) \widehat{\nabla} u^{\delta}_D(v)) &= 0 \qquad \text{in}\ \Omega,	\\
		u^{\delta}_D(v) &= f \qquad \text{on}\ \Sigma_N,\\
		u^{\delta}_D(v)&=0 \qquad \text{on}\ \Sigma_D.
    \end{cases}
\end{equation*}
Note that, as $\varepsilon \to 0$, the phase-field variable $v$ attains mainly values close to $0$ and $1$, due to the fact that $\frac{1}{\varepsilon}\int_{\Omega}$v(1-v)\, dx prevails, with a smooth change between the two values in the zone around the interface of the cavity. 
The thickness of the interface is of order $\varepsilon$.
We show the existence of minima for the functional $J_{\delta,\varepsilon}$ and then we find the first necessary optimality condition for the relaxed optimization problem on which the reconstruction algorithm is based. In fact, we derive a robust iterative method similar to the one in \cite{DecEllSty16}, providing some numerical experiments. Numerically, we observe that minima of the functional $J_{\delta,\varepsilon}$ give an accurate approximation of the minima of $J_{reg}$, for $\delta$ and $\varepsilon$ sufficiently small. Analytical justifications of the convergence of the minima of $J_{\delta,\varepsilon}$ to those of $J_{reg}$ have not been studied in this paper. However, they will be subject of future researches.   

Kohn-Vogelius type functionals are widely applied in reconstruction algorithms for detection of cavities and inclusions, and for identification of unknown parameters \cite{EbeHarMefRez21,Kal18}. For instance, the following two groups of papers analyse the minimization of Kohn-Vogelius functionals, see \cite{BelMef13,BenJaiKhaZin17,BoucPeichSayeTouz17,CauDamKatTim13,DamHarPui19,MefZol15,Boch19} and \cite{CauConGod16,DeFLes15,GheHas21,HriHasAbdCho19,MenHriNov21}, making use of shape gradient and topological derivative techniques. We finally mention that the mathematical literature on reconstruction methods for elastic inclusions and cavities is always of remarkable interest thanks to the intimate connection with the industrial applications. Among the vast literature
on the subject, we refer the reader to \cite{AmeBurHac07,AmmBreEliGarJossKan15,Ammari,BC,AspBerSchMus20,CarRap08,DouCara20,EbeHar21,Ikehata11,Ikehata12,Kang03,Lesnic,CasFarGal12} to have an idea of the reconstruction techniques applied in this context.

The paper is organized as follows. 
In Section \ref{Notation.}, we recall some of the preliminaries definitions and results needed in the paper. In Section \ref{sec:prob formulation}, we introduce the mathematical problem and investigate continuity properties of the solution of the forward problems with respect to perturbations of the cavity in the Haussdorff topology. Then, we show the existence of minima for the Kohn-Vogelius functional $J_{reg}(C)$. In Section \ref{sec: phase field}, we approximate the cavity with an inclusion of small elastic tensor, studying the properties of the corresponding Kohn-Vogelius functional. Then, we introduce its phase-field relaxation, analyzing its differentiability properties and deriving the necessary optimality conditions related to the phase-field minimization problem. In Section \ref{sec:discretization}, we 
introduce the discretization of the forward problems and we propose the iterative reconstruction algorithm based on the optimality condition derived in the previous section, proving its convergence properties. In Section \ref{sec:algorithm}, we show the efficiency and robustness of our approach through some numerical experiments. In Section \ref{sec:conclusions}, we give some conclusions and provide some mathematical open problems.

\section{Notation, geometrical setting, and preliminaries}
\label{Notation.}
We introduce the needed notation and the functional setting for the analysis addressed in the paper.
From now on, we concentrate on the space dimensions $d=2,3$.

\subsubsection*{Notation.} We denote scalar quantities, points, and vectors in italics, e.g.  ${x}, {y}$ and ${u}, {v}$, 
and fourth-order tensors in blackboard face, e.g.  $\mathbb{A}, \mathbb{B}$. 

We denote with $\widehat{{A}}:=\tfrac{1}{2}\left({A}+{A}^T\right)$ the symmetric part of a second-order tensor ${A}$, where
${A}^T$ is the transpose matrix. 
Standard notation  is utilized for inner products for vectors and matrices, that is,
${u}\cdot {v}=\sum_{i} u_{i} v_{i}$, and  ${A}:
{B}=\sum_{i,j}a_{ij} b_{ij}$ ($B$ is a second-order tensor).
$|{A}|$ denotes the norm induced by the inner product on matrices:
	\begin{equation*}
		|{A}|=\sqrt{{A}:{A}}.
	\end{equation*}
\subsubsection*{Domains.} We need to represent locally a boundary as a graph of functions, hence we adopt the notation:  $\forall\, x\in\mathbb{R}^d$, we set $x=(x',x_d)$, where $x'\in\mathbb{R}^{d-1}$, $x_d\in\mathbb{R}$.
Given $r>0$, we denote by $B_{r}({x})\subset\mathbb{R}^d$ the set $B_{r}({x}):=\{(x',x_d)/\ |x'|^2+x_d^2<r^2\}$ and by $B'_{r}({x'})\subset\mathbb{R}^{d-1}$ the set $B'_{r}({x'}):=\{x'\in\mathbb{R}^{d-1}/\, |x'|^2<r^2\}$.
\begin{definition}[Lipschitz regularity of domains]\ \\
Let $\Omega$ be a bounded domain in $\mathbb{R}^d$. We say that a portion $\Sigma$ of $\partial \Omega$ is of Lipschitz class with constant $r_0$, $L_0$, if for any ${p}\in \Sigma$, there exists a rigid transformation of coordinates under which we have that ${p}$ is mapped to the origin  and
	\begin{equation*}
		\Omega\cap B_{r_0}({0})=\{{x}\in B_{r_0}({0})\, :\, x_d>\psi({x}')\},
		\end{equation*}
		where ${\psi}$ is a $C^{0,1}$ function on $B'_{r_0}({0})\subset \mathbb{R}^{d-1}$, such that
		\begin{equation*}
		\begin{aligned}
		{\psi}({0})&=0,\\
		\|{\psi}\|_{C^{0,1}(B'_{r_0}({0}))}&\leq L_0.
		\end{aligned}
		\end{equation*}
	\end{definition}
Given a bounded domain $\Omega$, we define 
\begin{equation}\label{eq:omega_d0}
\Omega^{d_0}=\{x\in\Omega\ / \ {dist}(x,\partial\Omega)\leq d_0\}. 
\end{equation}
In the sequel, we deal with the Hausdorff distance between two sets $\Omega_1$ and $\Omega_2$. For reader's convenience, we recall its definition:
    \begin{equation*}
        d_H(\Omega_1,\Omega_2)=\max\{\sup\limits_{x\in\Omega_1}\inf\limits_{y\in\Omega_2}\ dist(x,y),\sup\limits_{x\in\Omega_2}\inf\limits_{y\in\Omega_1}\ dist(x,y)\}.
    \end{equation*}

\subsubsection*{Functional setting.} Let $\Omega$ be a bounded domain. Given a function $v\in L^1(\Omega)$, we recall the definition of the total variation of $v$, that is  
\begin{equation}\label{eq:TV}
TV(v) = \sup \left\{ \int_\Omega v \text{div}(\varphi); \quad \varphi \in C^1_0(\Omega), \,  \|{\varphi}\|_{{L^\infty}(\Omega)}\leq 1 \right\},
\end{equation}
and of the $BV$ space, i.e., 
\begin{equation*}
BV(\Omega)= \{ v \in L^1(\Omega) \, : \, TV(v) < \infty \}.
\label{BV}
\end{equation*}
The BV space is endowed with its natural norm $\|{v}\|_{BV(\Omega)} = \|{v}\|_{L^1(\Omega)} + TV(v)$. \\
The perimeter of $\Omega$ is defined as
\begin{equation}\label{def:perimeter}
\textrm{Per}(\Omega)= TV(\chi_{\Omega}),
\end{equation}
where $\chi_{\Omega}$ is the characteristic function of the set $\Omega$.\\ Let $\partial\Omega:=\partial\Omega_0\cup\partial\Omega_1$. For the well-posedness of the boundary value problems involved in the paper, we need to utilize the following classical Sobolev spaces: 
\begin{equation*}
H^1_{0}(\Omega):=\{\upsilon\in H^1(\Omega): \upsilon\lfloor_{\partial\Omega}=0\},\qquad \textrm{and}\qquad H^1_{\partial\Omega_0}(\Omega):=\{\upsilon\in H^1(\Omega): \upsilon\lfloor_{\partial\Omega_0}=0\}, 
\end{equation*}
Finally, we recall the following inequalities, see for example \cite{AleMorRos08}.
\begin{proposition}\label{proposition1}
Let $\Omega$ be a bounded Lipschitz domain. For every $\upsilon\in H^1_{0}(\Omega)$ (or $\upsilon\in H^1_{\partial\Omega_0}(\Omega)$), there exists a positive constant $\overline{c}$, depending only on the Lipschitz constants of $\Omega$, such that
\begin{equation}\label{poincare}
\|\upsilon\|_{H^1(\Omega)} \leq \overline{c}\  \|\nabla \upsilon\|_{L^2(\Omega)}\qquad (\textrm{Poincar\'e inequality}).
\end{equation}
\begin{equation}\label{korn}
\|\nabla \upsilon \|_{L^2(\Omega)} \leq \overline{c}\ \|\widehat \nabla \upsilon\|_{L^2(\Omega)}\qquad (\textrm{Korn inequality}).
\end{equation}
\end{proposition}

\section{Problem formulation - a Kohn-Vogelius approach}\label{sec:prob formulation}
In this paper, we deal with the geometrical inverse problem of identification of cavities in an elastic body $\Omega\subset \mathbb{R}^d$, with $d=2,3,$ by boundary measurements, given by tractions and displacements. The reconstruction procedure is based on a phase field approach applied to a Kohn-Vogelius type functional. \\
We assume that $\Omega$ is a bounded domain with Lipschitz boundary, with constants $r_0$ and $L_0$, and $\partial\Omega:=\Sigma_N\cup \Sigma_D$, where $|\Sigma_N|$, $|\Sigma_D|>0$ and $\Sigma_D$ closed. \\
In the presence of a cavity $C$, we consider the following mixed boundary value problem 
\begin{equation}\label{eq:neumann_problem}
    \begin{cases}
    \textrm{div}(\mathbb{C}_0\widehat{\nabla}u_N)=0 & \textrm{in}\ \Omega\setminus C\\
    (\mathbb{C}_0\widehat{\nabla}u_N)n=0 & \textrm{on}\ \partial C\\
    (\mathbb{C}_0\widehat{\nabla}u_N)\nu=g & \textrm{on}\ \Sigma_N\\
    u_N=0 &\textrm{on}\ \Sigma_D,
    \end{cases}
\end{equation}
where $n, \nu$ are the unit outer normal vectors to $C, \Sigma_N$, respectively, $\mathbb{C}_0$ is a fourth order elastic tensor, and $\widehat{\nabla} u_N$ represents the deformation tensor.
We introduce the needed assumptions on the elastic tensor, the cavity and the boundary data.

\begin{assumption}\label{ass:elasticity_tensor}
$\mathbb{C}_0=\mathbb{C}_0(x)$ is a fourth-order uniformly bounded tensor which satisfies minor and major symmetries, that is  
	\begin{equation*}
	(\mathbb{C}_0)_{ijkh}(x)=(\mathbb{C}_0)_{jikh}(x)=(\mathbb{C}_0)_{khij}(x),\qquad  \forall 1\leq i,j,k,h\leq d,\ \textrm{and}\ \ {x}\in\Omega.
	\end{equation*}
As usual, we also assume that $\mathbb{C}_0$ is uniformly strongly convex, that is, $\mathbb{C}_0$ defines a positive-definite quadratic form on symmetric matrices: for $\xi_0>0$
	\begin{equation*}
	\mathbb{C}_0(x)\widehat{A}:\widehat{A}\geq \xi_0 
	|\widehat{A}|^{2}, \qquad \textrm{a.e in}\,\, \Omega.
	\end{equation*}
\end{assumption}
\begin{remark}
The tensor $\mathbb{C}_0$ is assumed to be defined in $\Omega$, and not only in $\Omega\setminus C$, because we develop a reconstruction algorithm based on the strategy of filling the cavity with a fictitious elastic material, as is often applied in the context of optimization problems (see, for example, \cite{BouCha03}).
\end{remark}
\begin{assumption}\label{ass:neumann_data} 
The Neumann boundary data
\begin{equation}
g\in L^2(\Sigma_N).    
\end{equation}
\end{assumption}

The cavity, denoted by $C$, satisfies the following properties.
\begin{assumption}\label{ass:cavity}
Let $C\in\mathcal{C}$, where 
\begin{center}
$\mathcal{C}$:=\{$C \subset \overline{\Omega}:$ \hbox{ compact, simply connected} $\partial C\in C^{0,1}$ with constant $r_0$, $L_0$\, and ${dist}(C,\partial\Omega)\geq 2d_0>0$\}.
\end{center}
\end{assumption}
\begin{remark}\label{rem:compactness_sets}
The class $\mathcal{C}$ is compact with respect to the Hausdorff topology, see for example \cite[Theorem 2.4.10]{HenPie18}, and also \cite{DalMaso93,MeRon13}.
\end{remark}

We emphasize that the choice of Lipschitz regularity is a standard assumption in geometrical inverse problems related to identification of cavities, see for example \cite{MorRos03,MorRos04}. In fact, in this setting, it is possible to show uniqueness for the inverse problem. 
\begin{remark}
From now on, we will denote with $c$ any constant possibly depending on $\Omega$, $r_0$, $L_0$, $d$, $\xi_0$, $d_0$, $\overline{c}$, and on the uniform bounds of the elasticity tensor.
\end{remark}
Existence and uniqueness of a weak solution in $H^1_{\Sigma_D}(\Omega\setminus C)$ for the problem \eqref{eq:neumann_problem} is a classical result and follows from an application of the Lax-Milgram theorem to the following weak formulation of \eqref{eq:neumann_problem}: Find $u_N\in H^1_{\Sigma_D}(\Omega\setminus C)$ solution to
    \begin{equation}\label{eq:weak_form_prob_cavity}
        \int_{\Omega\setminus C} \mathbb{C}_0\widehat{\nabla}{u_N}:\widehat{\nabla}{\varphi}\, dx=\int_{\Sigma_N}g\cdot \varphi\, d\sigma(x), \qquad \forall \varphi\in H^1_{\Sigma_D}(\Omega\setminus C).
    \end{equation}
With the choice $\varphi=u_N$ in \eqref{eq:weak_form_prob_cavity}, and with an application of the strong convexity of the elastic tensor, {and the use of} Korn and Poincar\'e inequalities (see Proposition \ref{proposition1}) is simple to find that 
\begin{equation}\label{eq:coercivity}
\int_{\Omega\setminus C} \mathbb{C}_0\widehat{\nabla}{u_N}:\widehat{\nabla}{u_N}\, dx\geq c \|\widehat{\nabla}u_N\|^2_{L^2(\Omega\setminus C)}\geq c \|\nabla u_N\|^2_{L^2(\Omega\setminus C)}\geq c \|u_N\|^2_{H^1(\Omega\setminus C)}.
\end{equation}
At the same time, using a Cauchy-Schwarz inequality, we get
\begin{equation}\label{eq:continuity_functional}
    \Bigg|\int_{\Sigma_N}g\cdot u_N\, d\sigma(x)\Bigg| \leq \|g\|_{L^2(\Sigma_N)}\|u_N\|_{L^2(\Sigma_N)}\leq c \|g\|_{L^2(\Sigma_N)}\|u_N\|_{H^1(\Omega\setminus C)}.
\end{equation}
Putting together the estimates \eqref{eq:coercivity} and \eqref{eq:H1_estimate}, we find the standard $H^1-$estimate of the solution of \eqref{eq:weak_form_prob_cavity}, that is 
\begin{equation}\label{eq:H1_estimate}
    \|u_N\|_{H^1(\Omega\setminus C)}\leq c \|g\|_{L^2(\Sigma_N)}.
\end{equation}
Note that $u_N\lfloor_{\partial\Omega}\in H^{1/2}(\partial\Omega)$, with $u_N\lfloor_{\Sigma_D}=0$ (by hypothesis) and $u_N\lfloor_{\Sigma_N}=f$.
\\
In this paper, we address the following problem.
\begin{problem}\label{pb:problem}
    \textit{Let Assumptions \ref{ass:elasticity_tensor}, \ref{ass:neumann_data}, and \ref{ass:cavity} hold. Given the Neumann datum $g$ and the measured displacement $f$ on the boundary $\Sigma_N$, identify and reconstruct the cavity $C$.}
\end{problem}
To this aim, we transform Problem \ref{pb:problem} into the following optimization problem
\begin{equation}\label{eq:optmization_pb_KV}
    \min\limits_{C\in\mathcal{C}}\ J_{KV}(C):=\frac{1}{2}\int_{\Omega\setminus C}\mathbb{C}_0\widehat{\nabla}{(u_N(C)-u_D(C))}:\widehat{\nabla}{(u_N(C)-u_D(C))}\, dx,
\end{equation}
where $J_{KV}$ is a Kohn-Vogelius type functional, and the states $u_N(C)$ and $u_D(C)$ are, respectively, solutions to \eqref{eq:neumann_problem} and
\begin{equation}\label{eq:dirichlet_problem}
     \begin{cases}
    \textrm{div}(\mathbb{C}_0\widehat{\nabla}u_D)=0 & \textrm{in}\ \Omega\setminus C\\
    (\mathbb{C}_0\widehat{\nabla}u_D)n=0 & \textrm{on}\ \partial C\\
    u_D=f & \textrm{on}\ \Sigma_N\\
    u_D=0 &\textrm{on}\ \Sigma_D.
    \end{cases} 
\end{equation}
\begin{remark}
The Kohn-Vogelius functional \eqref{eq:optmization_pb_KV} can be rewritten as 
\begin{equation*}
    J_{KV}(C)=J_N(C)+J_D(C)+J_{ND}(C),
\end{equation*}
where 
\begin{align}
    J_N(C)&=\frac{1}{2}\int_{\Omega\setminus C}\mathbb{C}_0\widehat{\nabla}{u_N(C)}:\widehat{\nabla}{u_N(C)}\, dx,\label{def:JN}\\ 
    J_D(C)&=\frac{1}{2}\int_{\Omega\setminus C}\mathbb{C}_0\widehat{\nabla}{u_D(C)}:\widehat{\nabla}{u_D(C)}\, dx\label{def:JD},\\
    J_{ND}(C)&=-\int_{\Omega\setminus C}\mathbb{C}_0\widehat{\nabla}{u_N(C)}:\widehat{\nabla}{u_D(C)}\, dx=-\int_{\Sigma_N}g\cdot f\, d\sigma(x)=:\overline{J}_{ND},\label{def:JND}
\end{align}
where in the last functional an integration by parts to $J_{ND}$ has been applied. Note that $\overline{J}_{ND}$ is a constant term independent on $C$. Therefore
\begin{equation}\label{eq:decomposition_KV}
    J_{KV}(C)=J_N(C)+J_D(C)+\overline{J}_{ND}.
\end{equation}

\end{remark}
\begin{remark}
Note that the weak variational solution $u_D$ of the problem \eqref{eq:dirichlet_problem} is the minimizer of the following energy functional
\begin{equation}\label{eq:energy_dirichlet}
    E(u,C)=\frac{1}{2}\int_{\Omega\setminus C}\mathbb{C}_0\widehat{\nabla}{u}:\widehat{\nabla}{u}\, dx,  
\end{equation}
that is 
\begin{equation}\label{eq:minimum_energy_dirichlet}
    J_D(C)=\min\limits_{\substack{u\in H^1(\Omega\setminus C)\\
    u=f\  \textrm{on}\  \Sigma_N, u=0\  \textrm{on}\  \Sigma_D}} E(u,C).
\end{equation}
\end{remark}
As a standard approach in inverse problems, we add to the functional in \eqref{eq:optmization_pb_KV} a regularization term (Tikhonov regularization). In this context, we introduce a penalization on the perimeter of the cavity $C$. Therefore, given a regularization parameter $\alpha>0$, we consider
\begin{equation}\label{eq:optmization_pb_reg}
    \min\limits_{C\in\mathcal{C}}\ J_{reg}(C):=J_{KV}(C)+\alpha\ \textrm{Per}(C),
\end{equation}
where $\textrm{Per}(C)$ is the perimeter of the set $C$ (see definition \eqref{def:perimeter}).

\subsection{Continuity of $J_{KV}$ with respect to $C$}
Thanks to the decomposition of $J_{KV}$ as in \eqref{eq:decomposition_KV}, in this section we show that the functionals $J_N(C)$ and $J_D(C)$ are continuous with respect to perturbations of the cavity $C$ in the Hausdorff distance. For, we apply the Mosco convergence which is one of the techniques applied in the optimization context to show continuity of solutions with respect to perturbations of domains. The continuity property of these functionals is the key step to prove the existence of a minimum for problem \eqref{eq:optmization_pb_reg}.\\  
First, we recall the definition of Mosco convergence with some of its properties. For more details, we refer the reader to \cite{BucBut05,HenPie18,Gia04,MeRon13} and references therein. 
\begin{definition}\label{def:Mosco conv}
Let $H$ be  a  
Hilbert space,  and $G_k$ a  sequence  of  closed subspaces of $H$, and $G$ a subset of $H$. It is said that $G_k$ converges in the sense of Mosco to $G$ if the following assertions hold
\begin{enumerate}[(i)]
    \item If $u_{k_j}\in G_{k_j}$ is such that $u_{k_j}\rightharpoonup u$ in $H$, then $u\in G$;\label{eq:Mosco_condition1}
    \item $\forall u\in G$, $\exists u_k\in G_k$ such that $u_k\to u$ in $H$.\label{eq:Mosco_condition2}
\end{enumerate}
\end{definition}
Given $\Omega$ and $\Omega\setminus C$, we can identify the Sobolev space $H^1(\Omega\setminus C)$ with a closed subspace of $L^2(\Omega,\mathbb{R}^{d+d^2})$ through the map
    \begin{equation}\label{eq:identification}
    \begin{aligned}
        H^{1}(\Omega\setminus C) &\hookrightarrow L^2(\Omega,\mathbb{R}^{d+d^2})\\
       u &\to (u,{\partial_{l}u_i}),\qquad \forall i,l=1,\cdots,d,
    \end{aligned}
    \end{equation}
where ${u}$ and ${\nabla u}$ are the extension to zero in $C$ of $u$ and $\nabla u$, respectively. 
Denoting by $C_k$ a sequence of sets in $\mathcal{C}$ {(see Assumption \ref{ass:cavity})} and with $u_k$ a sequence of functions in $H^1(\Omega\setminus C_k)$, we have that the same identification holds for $\Omega\setminus C_k$, extending $u_k$ and $\nabla u_k$ to zero in $C_k$. \\
Mosco convergence holds in the class of uniform Lipschitz domains, see for example \cite{BucBut05,Che75}, in fact we have the following result.
\begin{proposition}\label{th:Mosco_conv}
Let us assume that $C_k, C\subset \Omega$ belong to the class $\mathcal{C}$. If $C_k\to C$ in the Hausdorff metric, then $H^1(\Omega\setminus C_k)$ converges to $H^1(\Omega\setminus C)$ in the sense of Mosco. 
\end{proposition}
\begin{remark}
Mosco convergence holds also for Sobolev subspaces of $H^1(\Omega\setminus C)$, such as $H^1_0(\Omega\setminus C)$ and $H^1_{\Sigma_D}(\Omega\setminus C)$, see for example \cite{BucBut05}.
\end{remark}
We can now prove the continuity of the functionals $J_N(C)$ and $J_D(C)$ with respect to perturbations of the cavity $C$. 
\begin{proposition}\label{prop:continuity_neumann}
Consider a sequence $C_k\in\mathcal{C}$ converging to $C$ in the Hausdorff metric (cf. Remark \ref{rem:compactness_sets}). 
Let $u_{N,k}:=u_N(C_k)\in H^1_{\Sigma_D}(\Omega\setminus C_k)$ and $u_N:=u_N(C)\in H^1_{\Sigma_D}(\Omega\setminus C)$ be solutions of \eqref{eq:weak_form_prob_cavity} in $\Omega\setminus C_k$ and $\Omega\setminus C$, respectively. Then
\begin{equation}\label{eq:continuity_neumann}
    \int_{\Omega\setminus C}\mathbb{C}_0{\widehat{\nabla}{u_{N,k}}}:{\widehat{\nabla}{u_{N,k}}}\, dx \longrightarrow \int_{\Omega\setminus C}\mathbb{C}_0{\widehat{\nabla}{u_N}}:{\widehat{\nabla}{u_N}}\, dx\,\qquad \textrm{as}\ k\to+\infty  
\end{equation}
\end{proposition}
\begin{proof}
The first part of the proof of this proposition is completely analogous to the one in \cite[Theorem 2.5]{AspBerCavRocVer22}. After proving that $u_{N,k}\to u_N$ in $L^2(\Sigma_N)$, it follows that, as $k\to +\infty$, 
\begin{equation*}
\begin{aligned}
    \int_{\Omega\setminus C_k} \mathbb{C}_0\widehat{\nabla}{u_{N,k}}:\widehat{\nabla}{u_{N,k}}\, dx&=\int_{\Sigma_N}g\cdot u_{N,k}\, d\sigma(x)\longrightarrow\\ 
    &\longrightarrow \int_{\Sigma_N}g\cdot u_{N}\, d\sigma(x)=\int_{\Omega\setminus C} \mathbb{C}_0\widehat{\nabla}{u_{N}}:\widehat{\nabla}{u_{N}}\, dx,
\end{aligned}
\end{equation*}
that is the assertion.
\end{proof}

To prove the continuity of the functional $J_D(C)$, we use the fact that, by hypothesis, the cavity $C$ does not touch the boundary of $\Omega$, see Assumption \ref{ass:cavity}. Therefore, we define a partition of the unity of $\Omega$, that is two functions $\phi, \psi\in C^{\infty}(\overline{\Omega})$, such that $\phi(x)+\psi(x)=1$, for all $x\in\overline{\Omega}$, and 
\begin{equation*}
    \phi=
    \begin{cases}
    1 & \textrm{in}\ \Omega^{d_0/2}\\
    0 & \textrm{in}\ \Omega\setminus \Omega^{d_0} 
    \end{cases}
\qquad\textrm{and}\qquad 
    \psi=
    \begin{cases}
    1 & \textrm{in}\ \Omega\setminus \Omega^{d_0}\\
    0 & \textrm{in}\ \Omega^{d_0/2}. 
    \end{cases}
\end{equation*}
We need to consider a lifting operator of the trace of the solution of the Dirichlet problem on the boundary of $\Omega$. Specifically, since $u_{D{\lfloor_{\partial\Omega}}}\in H^{1/2}(\partial\Omega)$ (by construction) and the trace operator has a right continuous inverse on Lipschitz domains (see \cite{Gri11}), we construct
\begin{equation}\label{eq:uf}
u^f\in H^1(\Omega)\quad \textrm{such that}\quad u^f_{\lfloor_{\partial\Omega}}:=u_{D{\lfloor_{\partial\Omega}}}.   
\end{equation}
\begin{proposition}\label{prop:continuity_dirichlet}
Let $C_k, C\in\mathcal{C}$ such that $C_k\to C$ in the Hausdorff metric. Then, for all $u\in H^1(\Omega\setminus C)$ such that $u_{\lfloor_{\partial\Omega\setminus \partial C}}=u^f_{\lfloor_{\partial\Omega}}$ there exists a sequence $u_k\in H^1(\Omega\setminus C_k)$ such that $u_{k{\lfloor_{\partial\Omega\setminus \partial C_k}}}=u^f_{\lfloor_{\partial\Omega}}$ and 
\begin{equation*}
    \int_{\Omega}\mathbb{C}_0{\widehat{\nabla}{u_k}}: {\widehat{\nabla}{u_k}}\, dx \longrightarrow \int_{\Omega}\mathbb{C}_0{\widehat{\nabla}{u}}: {\widehat{\nabla}{u}}\, dx,\qquad \textrm{as}\ k\to +\infty.
\end{equation*}
\end{proposition}
\begin{proof}
The proof of the proposition is based on the application of the Mosco convergence. 
First, note that, thanks to Definition \ref{def:Mosco conv}, point \eqref{eq:Mosco_condition2}, of Mosco convergence, for all $u\in H^1(\Omega\setminus C)$ there exists a sequence $u^*_k\in H^1(\Omega\setminus C_k)$ such that $u^*_k \to u$ in $L^2(\Omega)$. At the same time, for all $u\in H^1(\Omega\setminus C)$, we have that $(u-u^f)\phi\in H^1_0(\Omega\setminus C)$, then there exists, applying again the Mosco convergence, a sequence $v_k\in H^1_0(\Omega\setminus C_k)$ such that $v_k\to (u-u^f)\phi$ in $L^2(\Omega)$. Therefore, we define
\begin{equation*}
    u_k:=\psi u^*_k + v_k + u^f\phi.
\end{equation*}
Note that $u_{k\lfloor_{\partial\Omega}}=u^f_{\lfloor_{\partial\Omega}}$. Moreover, thanks to the convergence of $u^*_k$ and $v_k$, we get that 
\begin{equation*}
    u_k \longrightarrow u\ \textrm{in}\ L^2(\Omega),\qquad \textrm{as}\ k\to +\infty.
\end{equation*}
Therefore, by construction, we have that ${\nabla u_k} \to {\nabla u}$ in $L^2(\Omega)$ (strongly), as $k\to +\infty$. Hence, thanks to the definition of the deformation tensor, it follows that ${\widehat{\nabla}{u_k}} \to {\widehat{\nabla}{u}}$ in $L^2(\Omega)$ (strongly). Finally, using the boundedness of the elastic tensor, see Assumption \ref{ass:elasticity_tensor}, and the strong convergence of the deformation tensor, the assertion of the proposition follows. 
\end{proof}
We can now prove the existence of a minimum for the functional \eqref{eq:optmization_pb_reg}.
\begin{theorem}
For all $\alpha>0$, the minimum problem \eqref{eq:optmization_pb_reg} related to the regularized functional $J_{reg}$ has at least one solution.
\end{theorem}
\begin{proof}
Let $C_k\in\mathcal{C}$ be a minimizing sequence for $J_{reg}$. Thanks to Remark \ref{rem:compactness_sets}, there exists a subsequence that converges to $C\in\mathcal{C}$. We show that $C$ is in fact a minimum for $J_{reg}$. \\
Firstly, from Proposition \ref{prop:continuity_dirichlet}, given $u$ solution of \eqref{eq:minimum_energy_dirichlet}, with cavity $C$, there exists a sequence, that we denote by $u_{k}$ such that
\begin{equation*}
    E(u,C)\leq \liminf\limits_{k\to+\infty} E(u_{k},C_k),
\end{equation*}
where $E$ is defined in \eqref{eq:energy_dirichlet}, hence, $J_D(C)\leq \liminf\limits_{k\to+\infty}J_D(C_k)$.
Secondly, by the lower semicontinuity of the perimeter functional (see, for example, \cite[Section 5.2.1, Theorem 1]{EvaGar15}), it holds
\begin{equation*}
    \textrm{Per}(C)\leq \liminf_{k\rightarrow\infty}\textrm{Per}(C_k).
\end{equation*}
Then, by Propositions \ref{prop:continuity_neumann}, we get
\begin{equation*}
\begin{aligned}
    J_{reg}(C)&=\overline{J}_{ND}+ J_N(C)+J_D(C)+\alpha\textrm{Per}(C) \\
    &\leq \overline{J}_{ND}+ \liminf_{k\rightarrow\infty}J_N(C_k)+ \liminf_{k\rightarrow\infty}J_D(C_k)+\alpha \liminf_{k\rightarrow\infty}\textrm{Per}(C_k)\\
    &\leq \liminf_{k\rightarrow\infty}(\overline{J}_{ND}+J_N(C_k)+J_D(C_k)+\alpha \textrm{Per}(C_k))=\lim_{k\rightarrow\infty}J_{reg}(C_k)=\inf_{C^{\sharp}\in \mathcal{C}}J_{reg}(C^{\sharp}). 
\end{aligned}
\end{equation*}
\end{proof}
\section{A phase field approach}\label{sec: phase field}
The aim of this section is to describe a phase-field relaxation of the functional $J_{reg}$, see \eqref{eq:optmization_pb_reg}, introduced in the previous section in order to overcome, from a numerical point of view, the non-differentiability of the functional.\\
To be more specific, we consider in \eqref{eq:optmization_pb_reg} an approximation of the perimeter by a Ginzburg-Landau type functional (\cite{BouCha03}). This approach is widely applied in optimization procedures. We refer the reader to \cite{AS21,ABCHDRR20,BonCavFreRiv21,BGHFS14,BGHHR16,CRBHRA19,GHHHK15,GLNS21} and references therein for some recent papers on phase-field approaches. 

In the inverse problem context, applications of a phase-field approach have been proposed in \cite{DecEllSty16,RinRon11,Rondi11} for a linear elliptic equation, in \cite{BerRatVer18,BCP2021} for a semilinear elliptic equation, and very recently in \cite{AspBerCavRocVer22} for the Lam\'e system and in \cite{LamYou20} for a quasilinear Maxwell system.   

Firstly, we introduce the space
\begin{equation*}
X_{0,1}:=\{v\in BV(\Omega)\,:\, v=\chi_{C} \, \hbox{ a.e. in }\Omega, \,C\in {\mathcal C}\},    
\end{equation*}
where $\chi_C$ is the indicator function of $C$. The space $X_{0,1}$ is endowed with its natural norm $\|{v}\|_{BV(\Omega)} = \|{v}\|_{L^1(\Omega)} + TV(v)$ {(see Section \ref{Notation.})}.
Using this setting, Problem \eqref{eq:optmization_pb_reg} can be rephrased in the following way 
\begin{equation}\label{eq:optmization_pb_v}
    \min\limits_{v\in X_{0,1}}\ J(v):=J_{KV}(v)+\alpha\ TV(v),
\end{equation}

In the sequel, we often use the following result related to compactness properties of the space $BV$. 
\begin{remark}\label{compactness}
As a consequence of compactness properties of $BV(\Omega)$, \cite[Theorem 3.23]{AFP2000}, any uniformly bounded sequence in $X_{0,1}$ has a subsequence  converging in $L^1(\Omega)$ to an element in $X_{0,1}$. In fact, let $v_k$ a sequence uniformly bounded in $X_{0,1}$, then there exists, possibly up to a subsequence, $v\in BV(\Omega)$ such that 
\begin{equation*}
    v_k\to v\ \ \textrm{in}\ \ L^1(\Omega) \Rightarrow v_k\to v \ \ \textrm{a.e. in}\ \ \Omega.
\end{equation*}
Using the fact that $v_k$ attains values $0$ and $1$ only, it follows that $v\in X_{0,1}$.
\end{remark}

We can now regularize the problem by using a common approach in optimization procedures, that is of filling the voids (cavities) with a fictitious material with a small elastic tensor: Let 
$\delta>0$ be sufficiently small. We define
\begin{equation}\label{eq:elasticity_tensor_inclusion}
\mathbb{C}^{\delta}(v)= \mathbb{C}_0 +  (\mathbb{C}_1 - \mathbb{C}_0)v,\quad \textrm{with}\quad \mathbb{C}_1:=\delta\mathbb{C}_0.
\end{equation}
Tensors $\mathbb{C}_0$ and $\mathbb{C}_1$ correspond to the elastic tensors of $\Omega \setminus C$ and $C$, respectively. Moreover, $\mathbb{C}^{\delta}(v)$ is strongly convex by using the Assumption \ref{ass:elasticity_tensor}, and the fact that $\delta$ is positive and small.
 
Then, we consider the following optimization problem.
\begin{problem}
Given $\delta>0$, find
\begin{equation}\label{eq:optmization_pb_delta}
    \min\limits_{v\in X_{0,1}}\ J_{\delta}(v):=J^{\delta}_{KV}(v)+\alpha\ TV(v),
\end{equation}
\end{problem}
\noindent
where, recalling the definition \eqref{def:JND} of $\overline{J}_{ND}$,
\begin{equation}\label{def:JN and JD delta}
    \begin{aligned}
    J^{\delta}_{KV}(v)=\overline{J}_{ND}&+J^{\delta}_N(v)+J^{\delta}_D(v),\quad \textrm{and}\\
    J^{\delta}_N(v)=\frac{1}{2}\int_{\Omega}\mathbb{C}^{\delta}(v)\widehat{\nabla}{u^{\delta}_N(v)}:\widehat{\nabla}{u^{\delta}_N(v)}\, dx,
    &\qquad
    J^{\delta}_D(v)=\frac{1}{2}\int_{\Omega}\mathbb{C}^{\delta}(v)\widehat{\nabla}{u^{\delta}_D(v)}:\widehat{\nabla}{u^{\delta}_D(v)}\, dx.
\end{aligned}    
\end{equation}
Functions $u^{\delta}_N$ and $u^{\delta}_D$ are solutions to the following problems (similarly to \eqref{eq:neumann_problem} and \eqref{eq:dirichlet_problem})
\begin{equation}\label{pb:neumann_incl}
    \begin{cases}
    	\textrm{div}(\mathbb{C}^{\delta}(v) \widehat{\nabla} u^{\delta}_N(v)) &= 0 \qquad \text{in}\ \Omega,	\\
		(\mathbb{C}^{\delta}(v) \widehat\nabla u^{\delta}_N(v)) \nu &= g \qquad \text{on}\ \Sigma_N,\\
		u^{\delta}_N(v)&=0 \qquad \text{on}\ \Sigma_D,
    \end{cases}
\end{equation}
and
\begin{equation}\label{pb:dirichlet_incl}
    \begin{cases}
    	\textrm{div}(\mathbb{C}^{\delta}(v) \widehat{\nabla} u^{\delta}_D(v)) &= 0 \qquad \text{in}\ \Omega,	\\
		u^{\delta}_D(v) &= f \qquad \text{on}\ \Sigma_N,\\
		u^{\delta}_D(v)&=0 \qquad \text{on}\ \Sigma_D.
    \end{cases}
\end{equation}

Similarly to \eqref{eq:neumann_problem} and \eqref{eq:weak_form_prob_cavity}, the Neumann problem \eqref{pb:neumann_incl} has the following weak formulation: Find $u^{\delta}_N(v)\in H^1_{\Sigma_D}(\Omega)$ solution to
\begin{equation}\label{eq:weak_form_neumann_incl}
        \int_{\Omega} \mathbb{C}^{\delta}(v)\widehat{\nabla}{u^{\delta}_N(v)}:\widehat{\nabla}{\varphi}\, dx=\int_{\Sigma_N}g\cdot \varphi\, d\sigma(x), \qquad \forall \varphi\in H^1_{\Sigma_D}(\Omega).
    \end{equation}
Well-posedness in $H^1_{\Sigma_D}(\Omega)$ of the Neumann problem \eqref{eq:weak_form_neumann_incl} follows by the Lax-Milgram theorem similarly as showed for Problem \eqref{eq:weak_form_prob_cavity}, and in addition, analogously to \eqref{eq:H1_estimate}, we have
\begin{equation}\label{eq:H1est extended pb}
    \|u^{\delta}_N(v)\|_{H^1(\Omega)}\leq c \|g\|_{L^2_{\Sigma_N}}.
\end{equation}
The weak formulation of the Dirichlet problem can be obtained by using the lifting term $u^f$ defined in \eqref{eq:uf}. In fact, we can define $w^{\delta}_D(v):=u^{\delta}_D(v)-u^f$ and consider the following weak formulation: find $w^{\delta}_D(v)\in H^1_0(\Omega)$ solution to
\begin{equation}\label{eq:weak_form_dirichlet_incl}
    \int_{\Omega} \mathbb{C}^{\delta}(v)\widehat{\nabla}{w^{\delta}_D(v)}:\widehat{\nabla}{\psi}\, dx=-\int_{\Omega} \mathbb{C}^{\delta}(v)\widehat{\nabla}{u^{f}}:\widehat{\nabla}{\psi}\, dx,\qquad \forall\psi\in H^1_0(\Omega).
\end{equation}
Well-posedness in $H^1_0(\Omega)$ of the Dirichlet problem \eqref{eq:weak_form_dirichlet_incl} follows by the Lax-Milgram theorem,  analogously to \eqref{eq:weak_form_neumann_incl}, and in addition
\begin{equation*}
    \|w^{\delta}_D(v)\|_{H^1(\Omega)}\leq c \|f\|_{H^{1/2}(\Sigma_N)},\quad \textrm{hence}\quad \|u^{\delta}_D(v)\|_{H^1(\Omega)}\leq c \|f\|_{H^{1/2}(\Sigma_N)}.
\end{equation*}

The proof of the existence of a minimum for \eqref{eq:optmization_pb_delta} is based on a continuity result of the functional $J^{\delta}_{KV}$ in $X_{0,1}$.

\begin{remark}
We often use the following simplified notation similar to the one applied in the previous section to denote sequences: $u^{\delta}_{N,k}:= u^{\delta}_N(v_k), u^{\delta}_N := u^{\delta}_N(v), u^{\delta}_{D,k}:= u^{\delta}_D(v_k), \,  u^{\delta}_D := u^{\delta}_D(v), \mathbb{C}^{\delta}_k := \mathbb{C}^{\delta}(v_k), \,\mathbb{C}^{\delta} := \mathbb{C}^{\delta}(v)$.
\end{remark}
\begin{proposition}\label{continuity}
The maps $v\to J^{\delta}_N(v)$ and $v\to J^{\delta}_D(v)$ are continuous in the $L^1$ topology.
\end{proposition}
\begin{proof}
Let $v_k$ be a sequence in $X_{0,1}$ be strongly convergent in $L^1(\Omega)$ to $v \in X_{0,1}$. We divide the proof into two cases.\ \\
\textit{Case 1: continuity of $J^{\delta}_N(v)$ with respect to $v$.}\\
Let us consider the weak formulation \eqref{eq:weak_form_neumann_incl} associated to $v$ and $v_k$, respectively, that is 
\begin{equation*}
\int_\Omega\mathbb{C}^{\delta} \widehat\nabla u^{\delta}_N: \widehat\nabla \varphi = \int_{\Sigma_N} g\cdot\varphi, \quad \forall \varphi \in H_{\Sigma_D}^1(\Omega),
\end{equation*}
\begin{equation*}
\int_\Omega\mathbb{C}^{\delta}_k \widehat\nabla u^{\delta}_{N,k}: \widehat\nabla \varphi = \int_{\Sigma_N} g\cdot\varphi, \quad \forall \varphi \in H_{\Sigma_D}^1(\Omega).
\end{equation*}
Subtracting the two equations, we get
\begin{equation*}
\int_\Omega\mathbb{C}^{\delta}_k \widehat\nabla (u^{\delta}_{N,k} -  u^{\delta}_N) : \widehat\nabla \varphi + \int_{\Omega}( \mathbb{C}^{\delta}_k - \mathbb{C}^{\delta}) \widehat\nabla u^{\delta}_N :\widehat\nabla \varphi =0, \quad \forall \varphi \in H_{\Sigma_D}^1(\Omega),
\end{equation*}
hence, choosing $\varphi=u^{\delta}_{N,k} -  u^{\delta}_N$, and applying the same argument to get $H^1-$ estimates as in \eqref{eq:H1est extended pb}, we find that
\begin{equation*}
 \|u^{\delta}_{N,k}-u^{\delta}_N\|_{H^1(\Omega)}\leq c \|(\mathbb{C}^{\delta}_k - \mathbb{C}^{\delta})\widehat{\nabla} u^{\delta}_N\|_{L^2(\Omega)}.
\end{equation*}
Note that $\mathbb{C}^{\delta}_k-\mathbb{C}^{\delta}=(\mathbb{C}_1-\mathbb{C}_0)(v_k-v)$. Since $v_k - v \to 0$ in $L^1(\Omega)$ as $k\rightarrow +\infty$, then, possibly up to a subsequence, $v_k -v \to 0$, a.e. in $\Omega$. Moreover, since the elastic tensor is uniformly bounded, see Assumption \ref{ass:elasticity_tensor}, the dominated convergence theorem implies that $\|(\mathbb{C}^{\delta}_k - \mathbb{C}^{\delta})\widehat{\nabla} u^{\delta}_N\|_{L^2(\Omega)}\to 0$. Therefore,
\begin{equation*}
    \|u^{\delta}_{N,k}-u^{\delta}_N\|_{H^1(\Omega)}\to 0,\, \qquad \textrm{as}\ k\to\infty.
\end{equation*}
From the trace theorem, it follows that $\|u^{\delta}_{N,k} -  u^{\delta}_N\|_{L^2(\Sigma_N)} \to 0$, as $k\rightarrow +\infty$, hence
\begin{equation*}
\begin{aligned}
   \int_\Omega\mathbb{C}^{\delta}_k \widehat\nabla u^{\delta}_{N,k} : \widehat\nabla u^{\delta}_{N,k}\, dx&= \int_{\Sigma_N}g\cdot u^{\delta}_{N,k}\, d\sigma(x)  \xrightarrow{k\rightarrow+\infty} \int_{\Sigma_N}g\cdot u^{\delta}_{N}\, d\sigma(x)\\
   &=\int_\Omega\mathbb{C}^{\delta} \widehat\nabla u^{\delta}_{N}: \widehat\nabla u^{\delta}_{N}\, dx,
\end{aligned}
\end{equation*}
that is the assertion.\\
\textit{Case 2: continuity of $J^{\delta}_D(v)$ with respect to $v$.}\\
The proof of the second case follows the same arguments applied to Case 1. Let us define $w^{\delta}_D=u^{\delta}_D-u^f$, where $u^f$ has been defined in \eqref{eq:uf}, solution to \eqref{eq:weak_form_dirichlet_incl}.
Writing the equation \eqref{eq:weak_form_dirichlet_incl} for $v_k$ and $v$, we get
\begin{equation*}
\begin{aligned}
    \int_{\Omega} \mathbb{C}^{\delta}_k\widehat{\nabla}{w^{\delta}_{D,k}}:\widehat{\nabla}\psi\, dx &=-\int_{\Omega} \mathbb{C}^{\delta}_k\widehat{\nabla}{u^{f}}:\widehat{\nabla}{\psi}\, dx,\qquad \forall\psi\in H^1_0(\Omega),\\
    \int_{\Omega} \mathbb{C}^{\delta}\widehat{\nabla}{w^{\delta}_{D}}:\widehat{\nabla}\psi\, dx&=-\int_{\Omega} \mathbb{C}^{\delta}\widehat{\nabla}{u^{f}}:\widehat{\nabla}{\psi}\, dx,\qquad \forall\psi\in H^1_0(\Omega).
\end{aligned}
\end{equation*}
Subtracting the two equations, and adding and subtracting suitable terms, we get 
\begin{equation}\label{eq:aux1}
    \int_{\Omega}\mathbb{C}^{\delta}_k\widehat{\nabla}{\left(w^{\delta}_{D,k}-w^{\delta}_{D}\right)}:\widehat{\nabla}{\psi}\, dx+\int_{\Omega}\left(\mathbb{C}^{\delta}_k-\mathbb{C}^{\delta}\right)\widehat{\nabla}{w^{\delta}_D}:\widehat{\nabla}{\psi}\, dx=-\int_{\Omega}\left(\mathbb{C}^{\delta}_k-\mathbb{C}^{\delta}\right)\widehat{\nabla}{u^f}:\widehat{\nabla}{\psi}\, dx.
\end{equation}
Choosing $\psi=w^{\delta}_{D,k}-w^{\delta}_{D}$ in the previous equation, we get, for the first integral term,
\begin{equation}\label{eq:aux1.1}
    \int_{\Omega}\mathbb{C}^{\delta}_k\widehat{\nabla}{\left(w^{\delta}_{D,k}-w^{\delta}_{D}\right)}:\widehat{\nabla}{\left(w^{\delta}_{D,k}-w^{\delta}_{D}\right)}\geq c \|w^{\delta}_{D,k}-w^{\delta}_{D}\|^2_{H^1(\Omega)}.
\end{equation}
For the other two integral terms in \eqref{eq:aux1}, we find that
\begin{equation}\label{eq:aux2}
\begin{aligned}
    -\int_{\Omega}\left(\mathbb{C}^{\delta}_k-\mathbb{C}^{\delta}\right)\widehat{\nabla}{w^{\delta}_D}:\widehat{\nabla}{\left(w^{\delta}_{D,k}-w^{\delta}_{D}\right)}\, dx&-\int_{\Omega}\left(\mathbb{C}^{\delta}_k-\mathbb{C}^{\delta}\right)\widehat{\nabla}{u^f}:\widehat{\nabla}{\left(w^{\delta}_{D,k}-w^{\delta}_{D}\right)}\, dx\\
    &=-\int_{\Omega}\left(\mathbb{C}^{\delta}_k-\mathbb{C}^{\delta}\right)\widehat{\nabla}{u^{\delta}_D}:\widehat{\nabla}{\left(w^{\delta}_{D,k}-w^{\delta}_{D}\right)}\, dx,
\end{aligned}
\end{equation}
where in the last equality we have used the fact that $u^{\delta}_D=w^{\delta}_D+u^f$. Therefore, estimating the term on the right-hand side of \eqref{eq:aux2}, we get
\begin{equation}\label{eq:aux3}
  \Big|\int_{\Omega}\left(\mathbb{C}^{\delta}_k-\mathbb{C}^{\delta}\right)\widehat{\nabla}{u^{\delta}_D}:\widehat{\nabla}{\left(w^{\delta}_{D,k}-w^{\delta}_{D}\right)}\, dx\Big| \leq c \|\left(\mathbb{C}^{\delta}_k-\mathbb{C}^{\delta}\right)\widehat{\nabla}{u^{\delta}_D}\|_{L^2(\Omega)} \|w^{\delta}_{D,k}-w^{\delta}_{D}\|_{H^1(\Omega)}.  
\end{equation}
Putting together \eqref{eq:aux1.1} and \eqref{eq:aux3}, we find
\begin{equation*}
    \|w^{\delta}_{D,k}-w^{\delta}_{D}\|_{H^1(\Omega)}\leq c \|\left(\mathbb{C}^{\delta}_k-\mathbb{C}^{\delta}\right)\widehat{\nabla}{u^{\delta}_D}\|_{L^2(\Omega)}.
\end{equation*}
As in Case 1, we have that the term on the right-hand side tends to zero.\\ 
Therefore $\|w^{\delta}_{D,k}- w^{\delta}_{D}\|_{H^1(\Omega)}\to 0$ hence, from the fact that $w^{\delta}_{D,k}=u^{\delta}_{D,k}-u^f$ and $w^{\delta}_{D}=u^{\delta}_{D}-u^f$, we have that $\|u^{\delta}_{D,k}- u^{\delta}_{D}\|_{H^1(\Omega)}\to 0$. \\
We can now show the continuity of the functional $J^{\delta}_D(v)$ with respect to $v$. Since
\begin{equation}
\begin{aligned}
        \int_{\Omega} &\mathbb{C}^{\delta}_k\widehat{\nabla}{w^{\delta}_{D,k}}:\widehat{\nabla}w^{\delta}_{D,k}\, dx \\
        &\hspace{1cm}=-\int_{\Omega} \mathbb{C}^{\delta}_k\widehat{\nabla}{u^{f}}:\widehat{\nabla}{w^{\delta}_{D,k}}\, dx
        \xrightarrow{k\rightarrow+\infty} -\int_{\Omega} \mathbb{C}^{\delta}\widehat{\nabla}{u^{f}}:\widehat{\nabla}{w^{\delta}_{D}}\\
        &\hspace{8cm}=\int_{\Omega} \mathbb{C}^{\delta}\widehat{\nabla}{w^{\delta}_{D}}:\widehat{\nabla}w^{\delta}_{D}\, dx, 
\end{aligned}
\end{equation}
from straightforward calculations, we find
\begin{equation*}
\begin{aligned}
    \int_{\Omega}\mathbb{C}^{\delta}_k\widehat{\nabla}{u^{\delta}_k}:\widehat{\nabla}{u^{\delta}_k}\, dx&=\int_{\Omega}\mathbb{C}^{\delta}_k\widehat{\nabla}{w^{\delta}_{D,k}} : \widehat{\nabla}{w^{\delta}_{D,k}}\, dx\\
    &\hspace{1cm}+2\int_{\Omega}\mathbb{C}^{\delta}_k \widehat{\nabla}{w^{\delta}_{D,k}}: \widehat{\nabla}{u^f}\, dx + \int_{\Omega}\mathbb{C}^{\delta}_k\widehat{\nabla}{u^f}:\widehat{\nabla}{u^f}\, dx\\
    &\hspace{1cm}\xrightarrow{k\rightarrow+\infty} \int_{\Omega}\mathbb{C}^{\delta}\widehat{\nabla}{w^{\delta}_{D}} : \widehat{\nabla}{w^{\delta}_{D}}\, dx\\
    &\hspace{1cm}+2\int_{\Omega}\mathbb{C}^{\delta} \widehat{\nabla}{w^{\delta}_{D}}: \widehat{\nabla}{u^f}\, dx + \int_{\Omega}\mathbb{C}^{\delta}\widehat{\nabla}{u^f}:\widehat{\nabla}{u^f}\, dx\\
    &\hspace{5cm}=\int_{\Omega}\mathbb{C}^{\delta}\widehat{\nabla}{u^{\delta}}:\widehat{\nabla}{u^{\delta}}\, dx, 
\end{aligned}
\end{equation*}
hence the continuity of the functional $J^{\delta}_D$. 
\end{proof}
{Using the same arguments in \cite{AspBerCavRocVer22}, it is straightforward to prove the following existence result.} 
\begin{proposition}
The functional $J_{\delta}(v)$ has at least a minimum $v \in X_{0,1}$.
\end{proposition}

\subsection{Modica-Mortola relaxation}
In this section, we consider a further regularization of the functional defined in \eqref{eq:optmization_pb_delta} in order to obtain a differentiable cost functional on a convex subspace of $H^1(\Omega)$, see for example \cite{DecEllSty16,BerRatVer18}.\\ 
Recalling \eqref{eq:omega_d0}, we define the convex set
\begin{equation*}
\mathcal{K} = \{ v \in H^1(\Omega): \ 0 \leq v(x) \leq 1 \ a.e. \ in \ \Omega, \ v(x) = 0 \ a.e. \ in \ \Omega^{d_0}\}.
\end{equation*}
For every $\varepsilon >0$, we replace the total variation term in \eqref{eq:optmization_pb_delta} with the Modica-Mortola relaxation (\cite{Mod87}), that is
\begin{problem}
Given $\delta,\varepsilon>0$, find
\begin{equation}
\min_{v \in \mathcal{K}} J_{\delta,\varepsilon}(v):= J^{\delta}_{KV}(v)
+ \gamma \!\int_{\Omega}\Big( \varepsilon|\nabla v|^2 + \frac{1}{\varepsilon}v(1-v)\Big)\, dx,\label{minrel}
\end{equation}
\end{problem}
\noindent
where $\gamma=\frac{4}{\pi}\alpha$, where $4/\pi=(2\int_0^1 \sqrt{v(1-v)}\, dv)^{-1}$ is a rescaling parameter (\cite{Alb96}) and $J^{\delta}_{KV}$ is defined in \eqref{def:JN and JD delta}.\\
The following result is completely analogous to the one in \cite{DecEllSty16,AspBerCavRocVer22,BerRatVer18}, {so we omit its proof.} 
\begin{proposition}\label{prop: existence_min_J_de}
For every $\varepsilon,\delta>0$, Problem \eqref{minrel} has a solution $v=v_{\delta,\varepsilon} \in {\cal K}$.
\end{proposition}

\subsection{Necessary optimality condition}
In this section, we find the first order necessary optimality condition related to the minimization problem \eqref{minrel}. For, we define 
\begin{equation}\label{eq:FN and FD}
\begin{aligned}
    F^{\delta}_{N}:\mathcal{K}&\to H^1(\Omega)\,\qquad\qquad &\textrm{and}&\qquad   F^{\delta}_{D}:&\mathcal{K}& &\to& H^1(\Omega),\\
    v&\to F^{\delta}_{N}(v)=u^{\delta}_N(v)\, &\textrm{and}&\qquad &v& &\to& F^{\delta}_{D}(v)=u^{\delta}_D(v),
\end{aligned}
\end{equation}
where $u^{\delta}_N(v)$ and $u^{\delta}_D(v)$ are solutions to \eqref{pb:neumann_incl} and \eqref{pb:dirichlet_incl}, respectively. Moreover, in the sequel, we use the set
\begin{equation}\label{eq:theta}
  \mathcal{K}-v=\{z \ s.t. \ z+v \in \mathcal{K}\}.
\end{equation}
In the following propositions, we first show that $F^{\delta}_{N}$ and $F^{\delta}_{D}$ are Frech\'et differentiable in $\mathcal{K}\subset L^{\infty}(\Omega)\cap H^1(\Omega)$. Then, we state and prove the theorem on the necessary optimality condition for $J_{\delta,\varepsilon}$.
\begin{proposition}\label{prop:FN}
The operator $F^{\delta}_{N}$, in \eqref{eq:FN and FD}, is Frech\'et differentiable in $\mathcal{K}$ and 
\begin{equation}\label{eq:derivative_FN}
    \left(F^{\delta}_N\right)'(v)[\vartheta]=\widetilde{u}^{\delta}_N(v),
\end{equation}
where $\vartheta$ is any of the elements of the set in \eqref{eq:theta} and $\widetilde{u}^{\delta}_N(v)$ is solution to
\begin{equation}\label{eq:aux4}
    \int_{\Omega}\mathbb{C}^{\delta}(v)\widehat{\nabla}{\widetilde{u}^{\delta}_N(v)}:\widehat{\nabla}{\varphi}\, dx=\int_{\Omega}\vartheta(\mathbb{C}_0-\mathbb{C}_1)\widehat{\nabla}{u^{\delta}_N(v)}:\widehat{\nabla}{\varphi}\, dx,\qquad \forall\varphi\in H^1_{\Sigma_D}(\Omega).
\end{equation}
\end{proposition}
\begin{proof}
Taking the weak variational formulation \eqref{eq:weak_form_neumann_incl} for $u^{\delta}_N(v+\vartheta)$, $u^{\delta}_N(v)$, we get that the difference $u^{\delta}_N(v+\vartheta)-u^{\delta}_N(v)$ satisfies
\begin{equation}\label{difference}
\begin{aligned}
&\int_{\Omega}\mathbb{C}^{\delta}(v + \vartheta) \widehat\nabla (u^{\delta}_N(v+\vartheta)-u^{\delta}_N(v)) : \widehat\nabla \varphi\, dx\\
+ &\int_{\Omega}(\mathbb{C}^{\delta}(v + \vartheta) - \mathbb{C}^{\delta}(v)) \widehat\nabla u^{\delta}_N(v) : \widehat\nabla \varphi\, dx  = 0, \quad \forall \varphi \in H_{\Sigma_D}^1(\Omega).
\end{aligned}
\end{equation}
Choosing $\varphi = u^{\delta}_N(v+\vartheta)-u^{\delta}_N(v)$ and recalling that $\mathbb{C}^{\delta}(v + \vartheta) - \mathbb{C}^{\delta}(v) = (\mathbb{C}_1 - \mathbb{C}_0)\vartheta$,  we obtain
\begin{equation*}
\begin{aligned}
&\int_{\Omega}\mathbb{C}^{\delta}(v +\vartheta) \widehat\nabla (u^{\delta}_N(v+\vartheta)-u^{\delta}_N(v)):\widehat\nabla (u^{\delta}_N(v+\vartheta)-u^{\delta}_N(v))\, dx \\
=
- &\int_{\Omega} \vartheta(\mathbb{C}_1 - \mathbb{C}_0)\widehat\nabla u^{\delta}_N(v) : \widehat\nabla (u^{\delta}_N(v+\vartheta)-u^{\delta}_N(v))\, dx.
\end{aligned}
\end{equation*}
By means of the assumptions on the elasticity tensors, see Assumption \ref{ass:elasticity_tensor}, Korn and Poincar\'e inequalities,
and since $v + \vartheta \in \mathcal{K}$, we find that
\begin{equation}\label{estdiff}
\begin{aligned}
&\|{u^{\delta}_N(v+\vartheta) - u^{\delta}_N(v)}\|_{H^1(\Omega)} \leq c\|{\vartheta}\|_{L^\infty(\Omega)}\|{u^{\delta}_N(v)}\|_{H^1(\Omega)}
\leq c\|{\vartheta}\|_{L^\infty(\Omega)}.
\end{aligned}
\end{equation}
Subtract \eqref{eq:aux4} from \eqref{difference}, hence for all $\varphi\in H^1_{\Sigma_D}(\Omega)$
\begin{equation}\label{eq:aux9}
    \int_{\Omega}\mathbb{C}^{\delta}(v+\vartheta)\widehat{\nabla}{(u^{\delta}_N(v+\vartheta)-u^{\delta}_N(v))}:\widehat{\nabla}{\varphi}\, dx=\int_{\Omega}\mathbb{C}^{\delta}(v)\widehat{\nabla}{\widetilde{u}^{\delta}_N(v)}:\widehat{\nabla}{\varphi}\, dx.
\end{equation}
In the previous equation, choosing $\omega^{\delta}_N:=u^{\delta}_N(v+\vartheta)-u^{\delta}_N(v)$ and adding and subtracting $\mathbb{C}^{\delta}(v)\widehat{\nabla}{\omega^{\delta}_N}:\widehat{\nabla}{\varphi}$, we get
\begin{equation*}
\begin{aligned}
    \int_{\Omega}\mathbb{C}^{\delta}(v)\widehat{\nabla}{(\omega^{\delta}_N-\widetilde{u}^{\delta}_N(v))}:\widehat{\nabla}{\varphi}\, dx&=-\int_{\Omega}\left(\mathbb{C}^{\delta}(v+\vartheta) - \mathbb{C}^{\delta}(v)\right)\widehat{\nabla}{\omega^{\delta}_N}:\widehat{\nabla}{\varphi}\, dx\\
    &=\int_{\Omega}\vartheta(\mathbb{C}_0-\mathbb{C}_1)\widehat{\nabla}{\omega^{\delta}_N}:\widehat{\nabla}{\varphi}\, dx,\quad \forall\varphi\in H^1_{\Sigma_D}(\Omega).
\end{aligned}    
\end{equation*}
Then, taking $\varphi=\omega^{\delta}_N-\widetilde{u}^{\delta}_N$, and using the estimate on $\omega^{\delta}_N$ given in \eqref{estdiff}, we find
\begin{equation*}
    \|\omega^{\delta}_N-\widetilde{u}^{\delta}_N\|_{H^1(\Omega)}\leq c\|\vartheta\|_{L^{\infty}(\Omega)}\|\omega^{\delta}_N\|_{H^1(\Omega)}\leq c \|\vartheta\|^2_{L^{\infty}(\Omega)},
\end{equation*}
hence the assertion.
\end{proof}

As corollary, we find the differentiability of the Neumann functional $J^{\delta}_N$.
\begin{corollary}\label{cor:der JN}
The functional $J^{\delta}_N$, as defined in \eqref{def:JN and JD delta}, is Frech\'et differentiable in $\mathcal{K}$ and
	\begin{equation}
	\left(J^{\delta}_N\right)'(v)[\vartheta]=-\frac{1}{2}\int_{\Omega}\vartheta(\mathbb{C}_1-\mathbb{C}_0)\widehat{\nabla}{u^{\delta}_N(v)}:\widehat{\nabla}{u^{\delta}_N(v)}\, dx, 
	\end{equation}
where $\vartheta$ is any of the elements of the set in \eqref{eq:theta} and ${u}^{\delta}_N(v)$ is solution to \eqref{pb:neumann_incl}.	
\end{corollary}
\begin{proof}
Let us consider 
\begin{equation*}
\begin{aligned}
    J^{\delta}_N(v+\vartheta)-J^{\delta}_N(v)&=\frac{1}{2}\int_{\Omega}\mathbb{C}^{\delta}(v+\vartheta)\widehat{\nabla}{F^{\delta}_N(v+\vartheta)}:\widehat{\nabla}{F^{\delta}_N(v+\vartheta)}\, dx\\ &-\frac{1}{2}\int_{\Omega}\mathbb{C}^{\delta}(v)\widehat{\nabla}{F^{\delta}_N(v)}:\widehat{\nabla}{F^{\delta}_N(v)}\, dx.
\end{aligned}
\end{equation*}
Adding and subtracting suitable quantities in the previous equation, we get
\begin{equation*}
 \begin{aligned}
    &J^{\delta}_N(v+\vartheta)-J^{\delta}_N(v)\\
    &=\frac{1}{2}\int_{\Omega}\left(\mathbb{C}^{\delta}(v+\vartheta)-\mathbb{C}^{\delta}(v)\right)\widehat{\nabla}{\left(F^{\delta}_N(v+\vartheta)-F^{\delta}_N(v)\right)}:\widehat{\nabla}{\left(F^{\delta}_N(v+\vartheta)-F^{\delta}_N(v)\right)}\, dx\\
    &+\frac{1}{2}\int_{\Omega}\mathbb{C}^{\delta}(v)\widehat{\nabla}{\left(F^{\delta}_N(v+\vartheta)-F^{\delta}_N(v)\right)}:\widehat{\nabla}{\left(F^{\delta}_N(v+\vartheta)-F^{\delta}_N(v)\right)}\, dx\\
    &+\frac{1}{2}\int_{\Omega}\left(\mathbb{C}^{\delta}(v+\vartheta)-\mathbb{C}^{\delta}(v)\right)\widehat{\nabla}{F^{\delta}_N(v)}:\widehat{\nabla}{F^{\delta}_N(v)}\, dx\\
    &+\int_{\Omega}\mathbb{C}^{\delta}(v+\vartheta)\widehat{\nabla}{\left(F^{\delta}_N(v+\vartheta)-F^{\delta}_N(v)\right)}:\widehat{\nabla}{F^{\delta}_N(v)}\, dx.
\end{aligned}   
\end{equation*}
Using $H^1-$estimates for $F^{\delta}_N(v+\vartheta)-F^{\delta}_N(v)$ in \eqref{estdiff}, we have that the first two terms of the previous equation behaves as $\|\vartheta\|^3_{L^{\infty}(\Omega)}$ and $\|\vartheta\|^2_{L^{\infty}(\Omega)}$, respectively. Moreover, by means of equation \eqref{eq:derivative_FN}, and recalling that $\mathbb{C}^{\delta}(v+\vartheta)-\mathbb{C}^{\delta}(v)=\vartheta(\mathbb{C}_1-\mathbb{C}_0)$, we get that
\begin{equation*}
\begin{aligned}
    \left(J^{\delta}_N\right)'(v)[\vartheta]&=\frac{1}{2}\int_{\Omega}\vartheta\left(\mathbb{C}_1-\mathbb{C}_0\right)\widehat{\nabla}{u^{\delta}_N(v)}:\widehat{\nabla}{u^{\delta}_N(v)}\, dx \\
    &+\int_{\Omega}\mathbb{C}^{\delta}(v)\widehat{\nabla}{\widetilde{u}^{\delta}_N(v)}:\widehat{\nabla}{u^{\delta}_N(v)}\, dx.
\end{aligned}
\end{equation*}
From \eqref{eq:aux4}, choosing $\varphi=u^{\delta}_N(v)$, we find
\begin{equation*}
    \left(J^{\delta}_N\right)'(v)[\vartheta]=-\frac{1}{2}\int_{\Omega}\vartheta\left(\mathbb{C}_1-\mathbb{C}_0\right)\widehat{\nabla}{u^{\delta}_N(v)}:\widehat{\nabla}{u^{\delta}_N(v)}\, dx,
\end{equation*}
that is the assertion.
\end{proof}

\begin{proposition}\label{prop:FD}
The operator $F^{\delta}_{D}$, in \eqref{eq:FN and FD}, is Frech\'et differentiable in $\mathcal{K}$ and 
\begin{equation}\label{eq:derivative_FD}
    \left(F^{\delta}_D\right)'(v)[\vartheta]=\widetilde{u}^{\delta}_D(v),
\end{equation}
where $\vartheta$ is any of the elements of the set in \eqref{eq:theta} and $\widetilde{u}^{\delta}_D(v)$ is solution to
\begin{equation}\label{eq:aux5}
    \int_{\Omega}\mathbb{C}^{\delta}(v)\widehat{\nabla}{\widetilde{u}^{\delta}_D(v)}:\widehat{\nabla}{\psi}\, dx=\int_{\Omega}\vartheta(\mathbb{C}_0-\mathbb{C}_1)\widehat{\nabla}{u^{\delta}_D(v)}:\widehat{\nabla}{\psi}\, dx,\qquad \forall\psi\in H^1_0(\Omega).
\end{equation}
\end{proposition}
\begin{proof}
The proof of this proposition is analogous to the one of Proposition \ref{prop:FN}.\\
Let us consider $w^{\delta}_D(v)=u^{\delta}_D(v)-u^f$ and the related weak formulation \eqref{eq:weak_form_dirichlet_incl}. Taking the weak variational formulation \eqref{eq:weak_form_dirichlet_incl} for $w^{\delta}_D(v+\vartheta)$ and $w^{\delta}_D(v)$, we get that the difference $w^{\delta}_D(v+\vartheta)-w^{\delta}_D(v)\equiv u^{\delta}_D(v+\vartheta)-u^{\delta}_D(v)$ satisfies
\begin{equation*}
\begin{aligned}
\int_{\Omega}\mathbb{C}^{\delta}(v + \vartheta) \widehat\nabla (w^{\delta}_D(v+\vartheta)-w^{\delta}_D(v)) : \widehat\nabla \psi\, dx
&+ \int_{\Omega}(\mathbb{C}^{\delta}(v + \vartheta) - \mathbb{C}^{\delta}(v)) \widehat\nabla w^{\delta}_D(v) : \widehat\nabla \psi\, dx \\
&\hspace{-4cm}= -\int_{\Omega}\left(\mathbb{C}^{\delta}(v+\vartheta)-\mathbb{C}^{\delta}(v)\right)\widehat{\nabla}{u^f}:\widehat{\nabla}{\psi}\, dx, \quad \forall \psi \in H_0^1(\Omega),
\end{aligned}
\end{equation*}
that is 
\begin{equation}\label{difference_dirichlet}
\begin{aligned}
&\int_{\Omega}\mathbb{C}^{\delta}(v + \vartheta) \widehat\nabla (w^{\delta}_D(v+\vartheta)-w^{\delta}_D(v)) : \widehat\nabla \psi\, dx\\
&= -\int_{\Omega}\left(\mathbb{C}^{\delta}(v+\vartheta)-\mathbb{C}^{\delta}(v)\right)\widehat{\nabla}{u^{\delta}_D(v)}:\widehat{\nabla}{\psi}\, dx, \quad \forall \psi \in H_0^1(\Omega).
\end{aligned}
\end{equation}
Choosing $\psi = w^{\delta}_D(v+\vartheta)-w^{\delta}_D(v)$ and recalling that $\mathbb{C}^{\delta}(v + \vartheta) - \mathbb{C}^{\delta}(v) = (\mathbb{C}_1 - \mathbb{C}_0)\vartheta$,  we obtain
\begin{equation*}
\begin{aligned}
&\int_{\Omega}\mathbb{C}^{\delta}(v +\vartheta) \widehat\nabla (w^{\delta}_D(v+\vartheta)-w^{\delta}_D(v)):\widehat\nabla (w^{\delta}_D(v+\vartheta)-w^{\delta}_D(v)) \\
=
- &\int_{\Omega} \vartheta(\mathbb{C}_1 - \mathbb{C}_0)\widehat\nabla u^{\delta}_D(v) : \widehat\nabla (w^{\delta}_D(v+\vartheta)-w^{\delta}_D(v)).
\end{aligned}
\end{equation*}
By means of the assumptions on the elasticity tensors, see Assumption \ref{ass:elasticity_tensor}, Korn and Poincar\'e inequalities,
and since $v + \vartheta \in \mathcal{K}$, we find that
\begin{equation}\label{estdiff_dirichlet}
\begin{aligned}
&\|{w^{\delta}_D(v+\vartheta) - w^{\delta}_D(v)}\|_{H^1(\Omega)} \leq c\|{\vartheta}\|_{L^\infty(\Omega)}\|{u^{\delta}_D(v)}\|_{H^1(\Omega)}
\leq c\|{\vartheta}\|_{L^\infty(\Omega)},
\end{aligned}
\end{equation}
hence
\begin{equation}\label{eq:aux15}
    \|{u^{\delta}_D(v+\vartheta) - u^{\delta}_D(v)}\|_{H^1(\Omega)} \leq c\|{\vartheta}\|_{L^\infty(\Omega)}.
\end{equation}
Subtract \eqref{eq:aux5} from \eqref{difference_dirichlet}, hence, for all $\psi\in H^1_0(\Omega)$
\begin{equation}\label{eq:aux11}
    \int_{\Omega}\mathbb{C}^{\delta}(v+\vartheta)\widehat{\nabla}{(w^{\delta}_D(v+\vartheta)-w^{\delta}_D(v))}:\widehat{\nabla}{\psi}\, dx=\int_{\Omega}\mathbb{C}^{\delta}(v)\widehat{\nabla}{\widetilde{u}^{\delta}_D(v)}:\widehat{\nabla}{\psi}\, dx.
\end{equation}
In the previous equation, choosing $\omega^{\delta}_D:=w^{\delta}_D(v+\vartheta)-w^{\delta}_D(v)$ and adding and subtracting $\mathbb{C}^{\delta}(v)\widehat{\nabla}{\omega^{\delta}_D}:\widehat{\nabla}{\psi}$, we get
\begin{equation*}
\begin{aligned}
    \int_{\Omega}\mathbb{C}^{\delta}(v)\widehat{\nabla}{(\omega^{\delta}_D-\widetilde{u}^{\delta}_D(v))}:\widehat{\nabla}{\psi}\, dx&=-\int_{\Omega}\left(\mathbb{C}^{\delta}(v+\vartheta) - \mathbb{C}^{\delta}(v)\right)\widehat{\nabla}{\omega^{\delta}_D}:\widehat{\nabla}{\psi}\, dx\\
    &=\int_{\Omega}\vartheta(\mathbb{C}_0-\mathbb{C}_1)\widehat{\nabla}{\omega^{\delta}_D}:\widehat{\nabla}{\psi}\, dx,\quad \forall\psi\in H^1_0(\Omega).
\end{aligned}    
\end{equation*}
Then, taking $\psi=\omega^{\delta}_D-\widetilde{u}^{\delta}_D$, and using the estimate on $\omega^{\delta}_D$ given in \eqref{estdiff_dirichlet}, we find
\begin{equation*}
    \|\omega^{\delta}_D-\widetilde{u}^{\delta}_D\|_{H^1(\Omega)}\leq c\|\vartheta\|_{L^{\infty}(\Omega)}\|\omega^{\delta}_D\|_{H^1(\Omega)}\leq c \|\vartheta\|^2_{L^{\infty}(\Omega)},
\end{equation*}
hence the assertion.
\end{proof}

As corollary, we prove the differentiability of the Dirichlet functional $J^{\delta}_D$.
\begin{corollary}\label{cor:der JD}
The functional $J^{\delta}_D$, as defined in \eqref{def:JN and JD delta}, is Frech\'et differentiable in $\mathcal{K}$ and
	\begin{equation}
	\left(J^{\delta}_D\right)'(v)[\vartheta]=\frac{1}{2}\int_{\Omega}\vartheta(\mathbb{C}_1-\mathbb{C}_0)\widehat{\nabla}{u^{\delta}_D(v)}:\widehat{\nabla}{u^{\delta}_D(v)}\, dx,
	\end{equation}
where $\vartheta$ is any of the elements of the set in \eqref{eq:theta} and ${u}^{\delta}_D(v)$ is solution to \eqref{pb:dirichlet_incl}.	
\end{corollary}
\begin{proof}
Let us consider 
\begin{equation*}
\begin{aligned}
    J^{\delta}_D(v+\vartheta)-J^{\delta}_D(v)&=\frac{1}{2}\int_{\Omega}\mathbb{C}^{\delta}(v+\vartheta)\widehat{\nabla}{F^{\delta}_D(v+\vartheta)}:\widehat{\nabla}{F^{\delta}_D(v+\vartheta)}\, dx\\ &-\frac{1}{2}\int_{\Omega}\mathbb{C}^{\delta}(v)\widehat{\nabla}{F^{\delta}_D(v)}:\widehat{\nabla}{F^{\delta}_D(v)}\, dx.
\end{aligned}
\end{equation*}
Adding and subtracting suitable quantities in the previous equation, we get
\begin{equation*}
 \begin{aligned}
    &J^{\delta}_D(v+\vartheta)-J^{\delta}_D(v)\\
    &=\frac{1}{2}\int_{\Omega}\left(\mathbb{C}^{\delta}(v+\vartheta)-\mathbb{C}^{\delta}(v)\right)\widehat{\nabla}{\left(F^{\delta}_D(v+\vartheta)-F^{\delta}_D(v)\right)}:\widehat{\nabla}{\left(F^{\delta}_D(v+\vartheta)-F^{\delta}_D(v)\right)}\, dx\\
    &+\int_{\Omega}\left(\mathbb{C}^{\delta}(v+\vartheta)-\mathbb{C}^{\delta}(v)\right)\widehat{\nabla}{\left(F^{\delta}_D(v+\vartheta)-F^{\delta}_D(v)\right)}:\widehat{\nabla}{F^{\delta}_D(v)}\, dx\\
    &+\frac{1}{2}\int_{\Omega}\mathbb{C}^{\delta}(v)\widehat{\nabla}{\left(F^{\delta}_D(v+\vartheta)-F^{\delta}_D(v)\right)}:\widehat{\nabla}{\left(F^{\delta}_D(v+\vartheta)-F^{\delta}_D(v)\right)}\, dx\\
    &+\frac{1}{2}\int_{\Omega}\left(\mathbb{C}^{\delta}(v+\vartheta)-\mathbb{C}^{\delta}(v)\right)\widehat{\nabla}{F^{\delta}_D(v)}:\widehat{\nabla}{F^{\delta}_D(v)}\, dx\\
    &+\int_{\Omega}\mathbb{C}^{\delta}(v)\widehat{\nabla}{\left(F^{\delta}_D(v+\vartheta)-F^{\delta}_D(v)\right)}:\widehat{\nabla}{F^{\delta}_D(v)}\, dx.\\
\end{aligned}   
\end{equation*}
Using $H^1-$estimates for $F^{\delta}_D(v+\vartheta)-F^{\delta}_D(v)$ of Proposition \ref{prop:FD}, we have that the first term of the previous equation behaves as $\|\vartheta\|^3_{L^{\infty}(\Omega)}$ while the second and the third behave as $\|\vartheta\|^2_{L^{\infty}(\Omega)}$. Moreover, by means of equation \eqref{eq:derivative_FD}, and recalling that $\mathbb{C}^{\delta}(v+\vartheta)-\mathbb{C}^{\delta}(v)=\vartheta(\mathbb{C}_1-\mathbb{C}_0)$, we get that
\begin{equation*}
\begin{aligned}
    \left(J^{\delta}_D\right)'(v)[\vartheta]&=\frac{1}{2}\int_{\Omega}\vartheta\left(\mathbb{C}_1-\mathbb{C}_0\right)\widehat{\nabla}{u^{\delta}_D(v)}:\widehat{\nabla}{u^{\delta}_D(v)}\, dx \\
    &+\int_{\Omega}\mathbb{C}^{\delta}(v)\widehat{\nabla}{\widetilde{u}^{\delta}_D(v)}:\widehat{\nabla}{u^{\delta}_D(v)}\, dx.
\end{aligned}
\end{equation*}
Integrating by parts the last term, we find that 
\begin{equation*}
\begin{aligned}
    \int_{\Omega}\mathbb{C}^{\delta}(v)\widehat{\nabla}{\widetilde{u}^{\delta}_D(v)}:\widehat{\nabla}{u^{\delta}_D(v)}\, dx&=\int_{\Sigma_N}(\mathbb{C}^{\delta}(v)\widehat{\nabla}{u^{\delta}_D(v)})\nu \cdot \widetilde{u}^{\delta}_D(v)\, dx\\ &-\int_{\Omega}\textrm{div}\left(\mathbb{C}^{\delta}(v)\widehat{\nabla}{u^{\delta}_D(v)}\right)\cdot \widetilde{u}^{\delta}_D(v)\, dx\equiv 0
\end{aligned}
\end{equation*}
where we have used the fact that $\widetilde{u}^{\delta}_D(v)\in H^1_0(\Omega)$ and $u^{\delta}_D$ is solution to \eqref{pb:dirichlet_incl}. 
\end{proof}
Finally, we provide the optimality condition satisfied by the minimizers of $J_{\delta,\varepsilon}$.
\begin{theorem}\label{th:OC}
Any minimizer $v_{\delta,\varepsilon}$  of $J_{\delta,\varepsilon}$ satisfies the variational inequality
	\begin{equation}
	J'_{\delta,\varepsilon}(v_{\delta,\varepsilon})[\omega - v_{\delta,\varepsilon}] \geq 0, \qquad \forall \omega \in \mathcal{K},
	\label{OC}
	\end{equation}
	where
	\begin{equation}\label{VI}
	\begin{aligned}
		J'_{\delta,\varepsilon}(v)[\vartheta] =& \frac{1}{2}\int_{\Omega}\vartheta(\mathbb{C}_1-\mathbb{C}_0)\widehat{\nabla}{u^{\delta}_D}(v):\widehat{\nabla}{u^{\delta}_D}(v)\, dx
		-\frac{1}{2}\int_{\Omega}\vartheta(\mathbb{C}_1-\mathbb{C}_0)\widehat{\nabla}{u^{\delta}_N}(v):\widehat{\nabla}{u^{\delta}_N}(v)\, dx\\
&+  2 \gamma \varepsilon \int_{\Omega}\widehat\nabla v : \widehat\nabla \vartheta
+ \frac{\gamma}{\varepsilon}\int_{\Omega}(1-2v)\vartheta
\end{aligned}	
	\end{equation}
where $\vartheta$ is any of the elements of the set in \eqref{eq:theta}, and ${u}^{\delta}_D(v)$, ${u}^{\delta}_N(v)$ are solutions to \eqref{pb:dirichlet_incl} and \eqref{pb:neumann_incl}, respectively.	
\end{theorem}
\begin{proof}
The statement of the theorem follows putting together the results in Corollaries \ref{cor:der JN}, \ref{cor:der JD} and simply calculating the Frech\'et derivative of the Modica-Mortola functional. 
Since $J_{\delta,\varepsilon}$ is a continuous and Frech\'{e}t differentiable functional on a convex subset $\mathcal{K}$ of the Hilbert space $H^1(\Omega)$, the optimality conditions for the optimization problem \eqref{minrel} are expressed in terms of the variational inequality \eqref{OC}.
\end{proof}

\section{Reconstruction algorithm}
\label{sec:discretization} 
For numerical purposes, from now on, we assume a polygonal or polyhedral domain $\Omega$.

We denote with $(\mathcal{T}_h)_{0<h\leq h_0}$ a regular triangulation of $\Omega$ and we define 
\begin{equation}\label{calV}
\mathcal{V}_h := \{ v_h \in C(\overline\Omega) \, : \, v_h|_\mathcal{T} \in \mathcal{P}_1(\mathcal{T}), \, \forall \,  \mathcal{T} \in \mathcal{T}_h\}, 
    \end{equation}
where $\mathcal{P}_1(\mathcal{T})$ is the set of polynomials of degree one on $\mathcal{T}$. 
Then, let us denote by
\begin{equation}\label{calKh}
\begin{aligned}
\mathcal{V}_{h,\Sigma_D}:=\mathcal{V}_{h}\cap H^1_{\Sigma_D}(\Omega)&,\qquad \textrm{and}\quad \mathcal{V}_{h,0}:=\mathcal{V}_{h}\cap H^1_{0}(\Omega),\quad \textrm{and}\quad
\mathcal{K}_h &:=  \mathcal{V}_h \cap\mathcal{K}.
\end{aligned}
\end{equation}
For every $h>0$, we represent the discretized version of the solutions of problems \eqref{eq:weak_form_neumann_incl} and \eqref{eq:weak_form_dirichlet_incl} by $u^{\delta}_{N,h}(v_h)$ and $u^{\delta}_{D,h}(v_h)$, respectively.  More specifically, $u^{\delta}_{N,h}(v_h)$ is solution to 
\begin{equation}\label{eq:u_h}
\int_\Omega\mathbb{C}^{\delta}(v_h) \widehat\nabla u^{\delta}_{N,h}(v_h): \widehat\nabla \varphi_h\, dx = \int_{\Sigma_N} g_h\cdot\varphi_h\, d\sigma(x), \quad \forall \varphi_h \in \mathcal{V}_{h,\Sigma_D},
\end{equation}
where $g_h$ is a piecewise linear, continuous approximation of $g$, such that $g_h\to g$ in $L^2(\Sigma_N)$ as $h\to 0$. For the Dirichlet problem, we use a piecewise linear, continuous approximation $u^f_h$ of the lifting term $u^f$,  defined in \eqref{eq:uf}, assuming that  $u^f_h\to u^f$ in $H^1(\Omega)$, as $h\to 0$. For every $h>0$, we define $w^{\delta}_{D,h}:\mathcal{K}_h\to \mathcal{V}_{h,0}$, $w^{\delta}_{D,h}(v_h):=u^{\delta}_{D,h}(v_h)-u^f_h$, where $w^{\delta}_{D,h}(v_h)$ is solution to    
\begin{equation}\label{eq:w_h}
    \int_{\Omega}\mathbb{C}^{\delta}(v_h)\widehat{\nabla}{w^{\delta}_{D,h}(v_h)}:\widehat{\nabla}{\psi_h}\, dx=- \int_{\Omega}\mathbb{C}^{\delta}(v_h)\widehat{\nabla}{u^f_h}:\widehat{\nabla}{\psi_h}\, dx,\qquad \forall \psi_h\in \mathcal{V}_{h,0}.
\end{equation}
We recall that for every $v\in\mathcal{K}$ there exists a sequence $v_h\in \mathcal{K}_h$ such that $v_h \to v$ in $H^1(\Omega)$, see for example \cite{DecEllSty16}. 
\begin{proposition}\label{prop_conv_uN e uD}
Let $h_k$, $v_{h_k}$ be two sequences such that 
$h_k\to 0$, as $k\to+\infty$, and $v_{h_k}\in \mathcal{K}_{h_k}$ with $v_{h_k}\to v$ in $L^1(\Omega)$. Then, as $k\to +\infty$,
\begin{equation*}
    u^{\delta}_{N,h_k}(v_{h_k})\to u^{\delta}_N(v),\qquad \textrm{in}\ H^1_{\Sigma_D}(\Omega),
\end{equation*}
and
\begin{equation*}
    u^{\delta}_{D,h_k}(v_{h_k})\to u^{\delta}_D(v),\qquad \textrm{in}\ H^1(\Omega),
\end{equation*}
with $u^{\delta}_{D,h_k}\lfloor_{\partial\Omega}=u^f_{h_k}\lfloor_{\partial\Omega} \to u^f\lfloor_{\partial\Omega}=u^{\delta}_{D}\lfloor_{\partial\Omega}$ in $L^2(\partial\Omega)$, as $k\to +\infty$.
\end{proposition}
\begin{proof}
For the sake of simplicity, let us denote by $v_k:=v_{h_k}$, $u^{\delta}_{N,k}:=u^{\delta}_{N,h_k}(v_k)$, $u^{\delta}_{D,k}:=u^{\delta}_{D,h_k}(v_k)$, $g_k:=g_{h_k}$, $u^f_k:=u^f_{h_k}$, and $\mathbb{C}^{\delta}_k:=\mathbb{C}^{\delta}(v_k)$. Since, by hypothesis, $v_{k}\to v$ in $L^1(\Omega)$, then it holds $v_k\to v$\, a.e. in $\Omega$. 
Consider the weak formulation for $u^{\delta}_{N,k}$ and $u^{\delta}_N$, that is
\begin{equation*}
    \int_{\Omega}\mathbb{C}^{\delta}_k\widehat{\nabla}{u^{\delta}_{N,k}}:\widehat{\nabla}{\varphi}\, dx=\int_{\Sigma_N}g_k\cdot \varphi\, d\sigma(x),\qquad \forall\varphi\in \mathcal{V}_{h,\Sigma_D},
\end{equation*}
and
\begin{equation*}
    \int_{\Omega}\mathbb{C}^{\delta}\widehat{\nabla}{u^{\delta}_{N}}:\widehat{\nabla}{\varphi}\, dx=\int_{\Sigma_N}g\cdot \varphi\, d\sigma(x),\qquad \forall\varphi\in H^1_{\Sigma_D}.
\end{equation*}
Note that, since $\mathcal{V}_{h,\Sigma_D}\subset H^1_{\Sigma_D}$, we have that the last equation holds also for all $\varphi\in \mathcal{V}_{h,\Sigma_D}$. Then, subtracting the two equations, we get
\begin{equation*}
\int_{\Omega}\mathbb{C}^{\delta}_k\widehat{\nabla}{u^{\delta}_{N,k}}:\widehat{\nabla}{\varphi}\, dx-\int_{\Omega}\mathbb{C}^{\delta}\widehat{\nabla}{u^{\delta}_N}:\widehat{\nabla}{\varphi}\, dx=\int_{\Sigma_N}(g_k-g)\cdot\varphi\, d\sigma(x),\,\qquad \forall\varphi\in \mathcal{V}_{h,\Sigma_D},    
\end{equation*}
hence
\begin{equation}\label{eq:aux6}
    \int_{\Omega}\mathbb{C}^{\delta}_k\widehat{\nabla}{(u^{\delta}_{N,k}-u^{\delta}_N)}:\widehat{\nabla}{\varphi}\, dx-\int_{\Omega}(\mathbb{C}^{\delta}-\mathbb{C}^{\delta}_k)\widehat{\nabla}{u^{\delta}_N}:\widehat{\nabla}{\varphi}\, dx=\int_{\Sigma_N}(g_k-g)\cdot \varphi\, d\sigma(x).
\end{equation}
Let us choose $\overline{u}_{N,k}\in \mathcal{V}_{h,\Sigma_D}$ such that $\overline{u}_{N,k}\to u^{\delta}_N$ in $H^1_{\Sigma_D}(\Omega)$. Adding and subtracting suitable terms in \eqref{eq:aux6}, we get, for all $\varphi\in\mathcal{V}_{h,\Sigma_D}$,
\begin{equation*}
\begin{aligned}
    \int_{\Omega}\mathbb{C}^{\delta}_k\widehat{\nabla}{(u^{\delta}_{N,k}-\overline{u}_{N,k})}:\widehat{\nabla}{\varphi}\, dx&=\int_{\Omega}\mathbb{C}^{\delta}_k\widehat{\nabla}{(u^{\delta}_N-\overline{u}_{N,k})}:\widehat{\nabla}{\varphi}\, dx\\
    &+ \int_{\Omega}(\mathbb{C}^{\delta}-\mathbb{C}^{\delta}_k)\widehat{\nabla}{u^{\delta}_N}:\widehat{\nabla}{\varphi}\, dx\\
    &+\int_{\Sigma_N}(g_k-g)\cdot \varphi\, d\sigma(x).
\end{aligned}
\end{equation*}
Hence, choosing $\varphi=u^{\delta}_{N,k}-\overline{u}_{N,k}$, we get
\begin{equation}\label{eq:aux7}
    \|u^{\delta}_{N,k}-\overline{u}_{N,k}\|_{H^1(\Omega)}\leq c\left[ \|u^{\delta}_N-\overline{u}_{N,k}\|_{H^1(\Omega)}+\|(\mathbb{C}^{\delta}-\mathbb{C}^{\delta}_k)\widehat{\nabla}{u^{\delta}_N}\|_{L^2(\Omega)}+\|g_k-g\|_{L^2(\Omega)}\right],
\end{equation}
where the constant $c$ is independent on $k$.
Note that, thanks to the fact that $v_k\to v$ a.e. in $\Omega$, by means of the Lebesgue dominated convergence theorem $\|(\mathbb{C}^{\delta}-\mathbb{C}^{\delta}_k)\widehat{\nabla}{u^{\delta}_N}\|_{L^2(\Omega)}\to 0$, as $k\to +\infty$. Then, using this result in \eqref{eq:aux7} together with the convergence of $\overline{u}_{N,k}\to u^{\delta}_N$ in $H^1(\Omega)$, and $g_k\to g$ in $L^2(\Sigma_N)$, we get that the right-hand side of \eqref{eq:aux7} goes to zero as $k\to +\infty$. Then, the first assertion of the theorem follows.\\
For $u^{\delta}_{D,k}$, we make analogous calculations which involve $w^{\delta}_{D,k}=u^{\delta}_{D,k}-u^f_k$. Writing the weak formulations for $w^{\delta}_{D,k}$ and $w^{\delta}_D$ and noticing that $\mathcal{V}_{h,0}\subset H^1_0(\Omega)$, we have that
\begin{equation*}
    \int_{\Omega}\mathbb{C}^{\delta}_k\widehat{\nabla}{w^{\delta}_{D,k}}:\widehat{\nabla}{\psi}\, dx=-\int_{\Omega}\mathbb{C}^{\delta}_k\widehat{\nabla}{u^f_k}:\widehat{\nabla}{\psi}\, dx,\qquad \forall\psi\in \mathcal{V}_{h,0},
\end{equation*}
and
\begin{equation*}
    \int_{\Omega}\mathbb{C}^{\delta}\widehat{\nabla}{w^{\delta}_{D}}:\widehat{\nabla}{\psi}\, dx=-\int_{\Omega}\mathbb{C}^{\delta}\widehat{\nabla}{u^f}:\widehat{\nabla}{\psi}\, dx,\qquad \forall\psi\in \mathcal{V}_{h,0}.
\end{equation*}
Subtracting the two previous equations, and then adding and subtracting suitable terms, we find
\begin{equation*}
\begin{aligned}
    &\int_{\Omega}\mathbb{C}^{\delta}_k\widehat{\nabla}{(w^{\delta}_{D,k}-w^{\delta}_D)}:\widehat{\nabla}{\psi}\, dx+\int_{\Omega}(\mathbb{C}^{\delta}_k-\mathbb{C}^{\delta})\widehat{\nabla}{w^{\delta}_D}:\widehat{\nabla}{\psi}\, dx\\
    &=-\int_{\Omega}\mathbb{C}^{\delta}_k\widehat{\nabla}{(u^f_k-u^f)}:\widehat{\nabla}{\psi}\, dx+ \int_{\Omega}(\mathbb{C}^{\delta}-\mathbb{C}^{\delta}_k)\widehat{\nabla}{u^f}:\widehat{\nabla}{\psi}\, dx,\qquad \forall\psi\in \mathcal{V}_{h,0}.
\end{aligned}
\end{equation*}
Let us choose $\overline{w}_{D,k}\to w^{\delta}_D$ in $H^1_0(\Omega)$, as $k\to +\infty$. Then, we get that the previous equation is equivalent to 
\begin{equation*}
\begin{aligned}
    \int_{\Omega}\mathbb{C}^{\delta}_k\widehat{\nabla}{(w^{\delta}_{D,k}-\overline{w}_{D,k})}:\widehat{\nabla}{\psi}\, dx=
    &-\int_{\Omega}\mathbb{C}^{\delta}_k\widehat{\nabla}{(\overline{w}_{D,k}-w^{\delta}_D)}:\widehat{\nabla}{\psi}\, dx\\
    &-\int_{\Omega}(\mathbb{C}^{\delta}_k-\mathbb{C}^{\delta})\widehat{\nabla}{w^{\delta}_D}:\widehat{\nabla}{\psi}\, dx\\
    &-\int_{\Omega}\mathbb{C}^{\delta}_k\widehat{\nabla}{(u^f_k-u^f)}:\widehat{\nabla}{\psi}\, dx\\
    &+ \int_{\Omega}(\mathbb{C}^{\delta}-\mathbb{C}^{\delta}_k)\widehat{\nabla}{u^f}:\widehat{\nabla}{\psi}\, dx,\qquad \forall\psi\in \mathcal{V}_{h,0}.
\end{aligned}
\end{equation*}
Choosing $\psi=w^{\delta}_{D,k}-\overline{w}_{D,k}$, we get that
\begin{equation*}
\begin{aligned}
    \|w^{\delta}_{D,k}-\overline{w}_{D,k}\|_{H^1(\Omega)}\leq& c\Big[\|\overline{w}_{D,k}-w^{\delta}_D\|_{H^1(\Omega)}+\|(\mathbb{C}^{\delta}_k-\mathbb{C}^{\delta})\widehat{\nabla}{w^{\delta}_D}\|_{L^2(\Omega)}\\
    &+\|u^f_k-u^f\|_{H^1(\Omega)}+\|(\mathbb{C}^{\delta}_k-\mathbb{C}^{\delta})\widehat{\nabla}{u^f}\|_{L^2(\Omega)}\Big],
\end{aligned}
\end{equation*}
where the constant $c$ is independent on $k$. Arguing as in the previous case, we get that the right-hand side of the last equation is going to zero, hence, since $\overline{w}_{D,k}\to w^{\delta}_D$ in $H^1_0(\Omega)$ we have $\|w^{\delta}_{D,k}-w^{\delta}_D\|_{H^1(\Omega)}\to 0$\, as $k\to 0$. Now, since $w^{\delta}_{D,k}=u^{\delta}_{D,k}-u^f_k$ and $w^{\delta}_D=u^{\delta}_D-u^f$, we have that
\begin{equation}
    \|w^{\delta}_{D,k}-w^{\delta}_D\|_{H^1(\Omega)}\geq \Big| \|u^{\delta}_{D,k}-u^{\delta}_D\|_{H^1(\Omega)}-\|u^f_k-u^f\|_{H^1(\Omega)}\Big|,
\end{equation}
hence, since $u^f_k\to u^f$, as $k\to +\infty$ and $w^{\delta}_{D,k}\to w^{\delta}_D$ in $H^1(\Omega)$, we find that $u^{\delta}_{D,k}\to u^{\delta}_D$ in $H^1(\Omega)$.
\end{proof}
Denote by $J_{\delta,\varepsilon,h}:\mathcal{K}_h\to \mathbb{R}$ the approximation of $J_{\delta,\varepsilon}$, defined in \eqref{minrel}, hence we consider the problem
\begin{equation}\label{eq:minimum_probl_discretized}
        \min\limits_{v_h\in\mathcal{K}_h}J_{\delta,\varepsilon,h}(v_h):=J^{\delta}_{N,h}(v_h)+J^{\delta}_{D,h}(v_h)+\overline{J}_{ND,h}+\gamma \int_{\Omega}\left(\varepsilon|\nabla v_h|^2+\frac{1}{\varepsilon}v_h(1-v_h)  \right)\, dx,
\end{equation}
where 
\begin{equation*}
    \begin{aligned}
    J^{\delta}_{N,h}(v_h)&=\frac{1}{2}\int_{\Omega}\mathbb{C}^{\delta}(v_h)\widehat{\nabla}{u^{\delta}_{N,h}(v_h)}:\widehat{\nabla}{u^{\delta}_{N,h}(v_h)}\, dx,\\
    J^{\delta}_{D,h}(v_h)&=\frac{1}{2}\int_{\Omega}\mathbb{C}^{\delta}(v_h)\widehat{\nabla}{u^{\delta}_{D,h}(v_h)}:\widehat{\nabla}{u^{\delta}_{D,h}(v_h)}\, dx,\\
    \overline{J}_{ND,h}&=\int_{\Sigma_N}f_h\cdot g_h\, d\sigma(x).
    \end{aligned}
\end{equation*}

\begin{theorem}\label{th:5.2}
For every $\delta, \varepsilon>0$, problem \eqref{eq:minimum_probl_discretized} has a solution $v_h\in \mathcal{K}_h$.\\
Moreover, let $h_k$ be a sequence such that $h_k\to 0$, as $k\to +\infty$. Then, any $v_{h_k}$ has a subsequence strongly convergent in $H^1(\Omega)$ and a.e. in $\Omega$ to a minimum of $J_{\delta,\varepsilon}$. 
\end{theorem} 
\begin{proof}
The existence of a minimum follows straightforwardly, thanks to the fact that the analysis is addressed in a finite-dimensional space.\\
We prove the second part of the statement. Let $v_k:=v_{h_k}\in \mathcal{K}_{h_k}$ a minimizing sequence for \eqref{eq:minimum_probl_discretized}. Hence, as in the continuous case, see Proposition \ref{prop: existence_min_J_de}, we get that $v_k$ is bounded in $H^1(\Omega)$, hence, there exists a subsequence (still denoted by $v_k$) such that $v_k\rightharpoonup v$ in $H^1(\Omega)$ and $v_k\to v$ in $L^2(\Omega)$. Then, it follows that $v_k\to v$ in $L^1(\Omega)$ and $v_k\to v$ a.e. in $\Omega$. Thanks to Proposition \ref{prop_conv_uN e uD}, we have that
\begin{equation*}
    u^{\delta}_{N,k}(v_k)\to u^{\delta}_N(v),\qquad \textrm{in}\ H^1_{\Sigma_D}(\Omega),
\end{equation*}
and
\begin{equation*}
    u^{\delta}_{D,k}(v_k)\to u^{\delta}_D(v),\qquad \textrm{in}\ H^1(\Omega),
\end{equation*}
with $u^{\delta}_{D,k}\lfloor_{\partial\Omega}\to u^{\delta}_{D}\lfloor_{\partial\Omega}$ in $L^2(\partial\Omega)$. Let us show that $v$ is a minimum for $J_{\delta,\varepsilon}$. Let $\eta\in\mathcal{K}$ be arbitrary and choose $\eta_k:=\eta_{h_k}\in \mathcal{K}_{h_k}$ such that $\eta_k\to \eta$ in $H^1(\Omega)$. Since $v_k$ is a minimizing sequence, we have that 
\begin{equation*}
    J_{\delta,\varepsilon,k}(v_k)\leq J_{\delta,\varepsilon,k}(\eta_k).
\end{equation*}
Thanks to the lower semicontinuity of the norm, we have that $\|\nabla v\|^2_{L^2(\Omega)}\leq \liminf\limits_{k\to+\infty}\|\nabla v_k\|^2_{L^2(\Omega)}$. Hence,
\begin{equation*}
    J_{\delta,\varepsilon}(v)\leq \liminf\limits_{k\to+\infty} J_{\delta,\varepsilon,k}(v_k)\leq \liminf\limits_{k\to+\infty} J_{\delta,\varepsilon,k}(\eta_k)=\lim\limits_{k\to+\infty}J_{\delta,\varepsilon,k}(\eta_k)=J_{\delta,\varepsilon}(\eta),
\end{equation*}
that is, since $\eta$ is arbitrary, $J_{\delta,\varepsilon}(v)=\inf\limits_{\eta\in\mathcal{K}}J_{\delta,\varepsilon}(\eta)$.
We now show that $v_k$ converges strongly to $v$ in $H^1(\Omega)$. Let $\overline{v}_k\in\mathcal{K}_{h_k}$ such that $\overline{v}_k\to v$ in $H^1(\Omega)$. Therefore,
\begin{equation*}
    J_{\delta,\varepsilon}(v)\leq \liminf\limits_{k\to\infty}J_{\delta,\varepsilon,k}(v_k)\leq \liminf\limits_{k\to\infty}J_{\delta,\varepsilon,k}(\overline{v}_k)=\lim\limits_{k\to+\infty}J_{\delta,\varepsilon,k}(\overline{v}_k)=J_{\delta,\varepsilon}(v),
\end{equation*}
that is $J_{\delta,\varepsilon}(v)=\lim\limits_{k\to+\infty}J_{\delta,\varepsilon,k}(v_k)$. \\
Finally, it is simple to show that $\|\nabla v_k\|^2_{L^2(\Omega)}\to \|\nabla v\|^2_{L^2(\Omega)}$. In fact, 
\begin{equation*}
\begin{aligned}
    \gamma\varepsilon\int_{\Omega}|\nabla v_k|^2\, dx&= J_{\delta,\varepsilon,k}(v_k)-\frac{\gamma}{\varepsilon}\int_{\Omega}v_k(1-v_k)\, dx\\
    &-\frac{1}{2}\int_{\Omega}\mathbb{C}^{\delta}(v_k)\widehat{\nabla}{u^{\delta}_{N,k}(v_k)}:\widehat{\nabla}{u^{\delta}_{N,k}(v_k)}\, dx\\
    &-\frac{1}{2}\int_{\Omega}\mathbb{C}^{\delta}(v_k)\widehat{\nabla}{u^{\delta}_{D,k}(v_k)}:\widehat{\nabla}{u^{\delta}_{D,k}(v_k)}\, dx - \int_{\Sigma_N}f_k\cdot g_k\, d\sigma(x).
\end{aligned}    
\end{equation*}
Thanks to the continuity results (Proposition \ref{prop_conv_uN e uD}) and the dominated convergence theorem, we get that the right-hand side of the previous equation goes to, as $k\to +\infty$,
\begin{equation*}
    \begin{aligned}
    J_{\delta,\varepsilon}(v)&-\frac{\gamma}{\varepsilon}\int_{\Omega}v(1-v)\, dx\\
    &-\frac{1}{2}\int_{\Omega}\mathbb{C}^{\delta}\widehat{\nabla}{u^{\delta}_{N}(v)}:\widehat{\nabla}{u^{\delta}_{N}(v)}\, dx\\
    &-\frac{1}{2}\int_{\Omega}\mathbb{C}^{\delta}\widehat{\nabla}{u^{\delta}_{D}(v)}:\widehat{\nabla}{u^{\delta}_{D}(v)}\, dx - \int_{\Omega}f\cdot g\, d\sigma(x)=\gamma\varepsilon\int_{\Omega}|\nabla v|^2\, dx.
\end{aligned}    
\end{equation*}
Therefore, $\|v_k-v\|_{H^1(\Omega)}\to 0$, as $k\to 0$.
\end{proof}
For the implementation of a numerical algorithm, we use the discretized version of the optimality condition \eqref{VI}, that is we search for $v_h\in\mathcal{K}_h$ satisfying
\begin{equation}\label{eq:discrete_opt_cond}
    J'_{\delta,\varepsilon,h}(v_h)[\omega_h-v_h]\geq 0,\qquad \forall\omega_h\in\mathcal{K}_h.
\end{equation}
Analogously to the continuous case, one can prove that 
\begin{equation*}
\begin{aligned}
    J'_{\delta,\varepsilon,h}(v_h)[\omega_h-v_h]&=\frac{1}{2}\int_{\Omega}(\omega_h-v_h)(\mathbb{C}_1-\mathbb{C}_0)\widehat{\nabla}{u^{\delta}_{D,h}(v_h)}: \widehat{\nabla}{u^{\delta}_{D,h}(v_h)}\, dx\\
    &-\frac{1}{2}\int_{\Omega}(\omega_h-v_h)(\mathbb{C}_1-\mathbb{C}_0)\widehat{\nabla}{u^{\delta}_{N,h}(v_h)}: \widehat{\nabla}{u^{\delta}_{N,h}(v_h)}\, dx\\
    &+2\gamma\varepsilon\int_{\Omega}\nabla v_h\cdot \nabla(\omega_h-v_h)\, dx +\frac{\gamma}{\varepsilon}\int_{\Omega}(1-2v_h)(\omega_h-v_h)\, dx\geq 0\, 
\end{aligned}
\end{equation*}
for all $\omega_h\in \mathcal{K}_h$.\\
Let us prove the following theorem.
\begin{theorem}
Let $h_k$ be a sequence such that $h_k\to 0$, as $k\to+\infty$, and $v_{h_k}$ a sequence satisfying \eqref{eq:discrete_opt_cond}. Then, there exists a subsequence of $v_{h_k}$ converging strongly in $H^1(\Omega)$ and a.e. in $\Omega$ to a solution $v$ of the continuous optimality condition \eqref{VI}. 
\end{theorem}
\begin{proof}
We use the following notation: $v_k:=v_{h_k}$, $u^{\delta}_{N,k}:=u^{\delta}_{N,h_k}(v_k)$, $u^{\delta}_{D,k}:=u^{\delta}_{D,h_k}(v_k)$, and $w^{\delta}_{D,k}=u^{\delta}_{D,k}-u^f_k$. Using the discretized weak formulation of $u^{\delta}_{N,k}$, see \eqref{eq:u_h}, and $w^{\delta}_{D,k}$, see \eqref{eq:w_h}, we get that
\begin{equation*}
    \|u^{\delta}_{N,k}\|_{H^1(\Omega)}\leq c\|g_k\|_{L^2(\Sigma_N)},\qquad \textrm{and}\qquad \|u^{\delta}_{D,k}\|_{H^1(\Omega)}\leq c\|f_k\|_{H^{1/2}(\Sigma_N)},
\end{equation*}
where $c$ is a constant independent on $k$. Consequently, from the discretized optimality condition \eqref{eq:discrete_opt_cond}, choosing $\omega_{h_k}=0$, and recalling that $v_k$ is uniformly bounded since $v_k\leq 1$, for all $k$, we get that
\begin{equation*}
    2\gamma\varepsilon\int_{\Omega}|\nabla v_k|^2\, dx\leq c\left[ \frac{1}{2}\|\nabla u^{\delta}_{N,k}\|^2_{L^2(\Omega)}+\frac{1}{2}\|\nabla u^{\delta}_{D,k}\|^2_{L^2(\Omega)}\right] + c_0(\Omega,\gamma,\varepsilon)\leq c,
\end{equation*}
where $c$ is independent of $k$.
Therefore, $v_k$ is uniformly bounded in $H^1(\Omega)$, hence there exists a subsequence (still denoted by $v_k$) and $v\in\mathcal{K}$ such that $v_k\rightharpoonup v$ in $H^1(\Omega)$, $v_k\to v$ in $L^2(\Omega)$, and $v_k\to v$ a.e. in $\Omega$. Thanks to Proposition \ref{prop_conv_uN e uD}, we have that $u^{\delta}_{N,k}(v_k)\to u^{\delta}_N(v)$ in $H^1_{\Sigma_D}(\Omega)$ and $u^{\delta}_{D,k}(v_k)\to u^{\delta}_D(v)$ in $H^1(\Omega)$ with $u^{\delta}_{D,k}\lfloor_{\partial\Omega}\to u^{\delta}_{D}\lfloor_{\partial\Omega}$. We have to show that $v$ satisfies the variational inequality \eqref{OC}. Let us choose $\omega\in\mathcal{K}$, then there exists $\overline{\omega}_k\in\mathcal{K}_k$ such that $\overline{\omega}_k\to \omega$ in $H^1(\Omega)$ and a.e. in $\Omega$.\\  
Consider the variational inequality
\begin{equation}\label{eq:aux8}
\begin{aligned}
    J'_{\delta,\varepsilon,k}(v_k)[\overline{\omega}_k-v_k]&=\frac{1}{2}\int_{\Omega}(\overline{\omega}_k-v_k)(\mathbb{C}_1-\mathbb{C}_0)\widehat{\nabla}{u^{\delta}_{D,k}(v_k)}: \widehat{\nabla}{u^{\delta}_{D,k}(v_k)}\, dx\\
    &-\frac{1}{2}\int_{\Omega}(\overline{\omega}_k-v_k)(\mathbb{C}_1-\mathbb{C}_0)\widehat{\nabla}{u^{\delta}_{N,k}(v_k)}: \widehat{\nabla}{u^{\delta}_{N,k}(v_k)}\, dx\\
    &+2\gamma\varepsilon\int_{\Omega}\nabla v_k\cdot \nabla(\overline{\omega}_k-v_k)\, dx +\frac{\gamma}{\varepsilon}\int_{\Omega}(1-2v_k)(\overline{\omega}_k-v_k)\, dx\geq 0.
\end{aligned}
\end{equation}
For example, taking the integral related to $u^{\delta}_{N,k}(v_k)$ in the previous equation, we get
\begin{equation*}
\begin{aligned}
    &\int_{\Omega}(\overline{\omega}_k-v_k)(\mathbb{C}_1-\mathbb{C}_0)\widehat{\nabla}{u^{\delta}_{N,k}(v_k)}: \widehat{\nabla}{u^{\delta}_{N,k}(v_k)}\, dx\\
    &=\int_{\Omega}(\overline{\omega}_k-v_k)(\mathbb{C}_1-\mathbb{C}_0)\widehat{\nabla}{(u^{\delta}_{N,k}(v_k)-u^{\delta}_{N}(v))}: \widehat{\nabla}{(u^{\delta}_{N,k}(v_k)-u^{\delta}_{N}(v))}\, dx\\
   &+ \int_{\Omega}(\overline{\omega}_k-v_k)(\mathbb{C}_1-\mathbb{C}_0)\widehat{\nabla}{(u^{\delta}_{N,k}(v_k)-u^{\delta}_{N}(v))}: \widehat{\nabla}{u^{\delta}_{N}(v)}\, dx\\
   &+ \int_{\Omega}(\overline{\omega}_k-v_k)(\mathbb{C}_1-\mathbb{C}_0)\widehat{\nabla}{u^{\delta}_{N}(v)}: \widehat{\nabla}{u^{\delta}_{N}(v)}\, dx\\
   &=\int_{\Omega}(\overline{\omega}_k-v_k)(\mathbb{C}_1-\mathbb{C}_0)\widehat{\nabla}{(u^{\delta}_{N,k}(v_k)-u^{\delta}_{N}(v))}: \widehat{\nabla}{(u^{\delta}_{N,k}(v_k)-u^{\delta}_{N}(v))}\, dx\\
   &+ \int_{\Omega}(\overline{\omega}_k-v_k)(\mathbb{C}_1-\mathbb{C}_0)\widehat{\nabla}{(u^{\delta}_{N,k}(v_k)-u^{\delta}_{N}(v))}: \widehat{\nabla}{u^{\delta}_{N}(v)}\, dx\\
   &+ \int_{\Omega}((\overline{\omega}_k-\omega)-(v_k-v))(\mathbb{C}_1-\mathbb{C}_0)\widehat{\nabla}{u^{\delta}_{N}(v)}: \widehat{\nabla}{u^{\delta}_{N}(v)}\, dx\\
   &+ \int_{\Omega}(\omega-v)(\mathbb{C}_1-\mathbb{C}_0)\widehat{\nabla}{u^{\delta}_{N}(v)}: \widehat{\nabla}{u^{\delta}_{N}(v)}\, dx.
\end{aligned}
\end{equation*}
Note that, the first and the second integral on the right-hand side of the last equation tend to zero thanks to the $L^{\infty}(\Omega)$ estimates on $\mathbb{C}_1$, $\mathbb{C}_0$, $\overline{\omega}_k$, $v_k$, and the $H^1-$estimates for $u^{\delta}_{N,k}(v_k)-u^{\delta}_N(v)$. The third term tends to zero thanks to the dominated convergence theorem. \\
The same arguments above apply to the first integral on the right-hand side of \eqref{eq:aux8} related to $u^{\delta}_{D,k}(v_k)$. Inserting these results in \eqref{eq:aux8}, and using the fact that $v_k\rightharpoonup v$ in $H^1(\Omega)$, hence $\|\nabla v\|^2_{L^2(\Omega)}\leq \liminf\limits_{k\to +\infty}\|\nabla v_k\|^2_{L^2(\Omega)}$, and noticing that $\int_{\Omega}v_k\overline{\omega}_k\, dx \to \int_{\Omega}v\omega\, dx$, as $k\to+\infty$, we get  
\begin{equation*}
\begin{aligned}
    &\frac{1}{2}\int_{\Omega}({\omega}-v)(\mathbb{C}_1-\mathbb{C}_0)\widehat{\nabla}{u^{\delta}_{D}(v)}: \widehat{\nabla}{u^{\delta}_{D}(v)}\, dx\\
    &-\frac{1}{2}\int_{\Omega}({\omega}-v)(\mathbb{C}_1-\mathbb{C}_0)\widehat{\nabla}{u^{\delta}_{N}(v)}: \widehat{\nabla}{u^{\delta}_{N}(v)}\, dx\\
    &+2\gamma\varepsilon\int_{\Omega}\nabla v\cdot \nabla({\omega}-v)\, dx +\frac{\gamma}{\varepsilon}\int_{\Omega}(1-2v)({\omega}-v)\, dx\\
    &\geq \liminf\limits_{k\to+\infty}\Bigg\{\frac{1}{2}\int_{\Omega}(\overline{\omega}_k-v_k)(\mathbb{C}_1-\mathbb{C}_0)\widehat{\nabla}{u^{\delta}_{D,k}(v_k)}: \widehat{\nabla}{u^{\delta}_{D,k}(v_k)}\, dx\\
    &-\frac{1}{2}\int_{\Omega}(\overline{\omega}_k-v_k)(\mathbb{C}_1-\mathbb{C}_0)\widehat{\nabla}{u^{\delta}_{N,k}(v_k)}: \widehat{\nabla}{u^{\delta}_{N,k}(v_k)}\, dx\\
    &+2\gamma\varepsilon\int_{\Omega}\nabla v_k\cdot \nabla(\overline{\omega}_k-v_k)\, dx +\frac{\gamma}{\varepsilon}\int_{\Omega}(1-2v_k)(\overline{\omega}_k-v_k)\, dx\Bigg\}\geq 0.
\end{aligned}
\end{equation*}
To conclude the proof, we have to show that $v_k\to v$ in $H^1(\Omega)$. Taking $\overline{v}_k\in\mathcal{K}_{h_k}$ such that $v_k\to v$ in $H^1(\Omega)$ and substituting $\omega_k=\overline{v}_k$ in \eqref{eq:aux8}, we find
\begin{equation*}
\begin{aligned}
    2\gamma\varepsilon\int_{\Omega}|\nabla v_k|^2\, dx&\leq 2\gamma\varepsilon\int_{\Omega}\nabla v_k\cdot \nabla \overline{v}_k\, dx +\frac{\gamma}{\varepsilon}\int_{\Omega}(1-2v_k)(\overline{v}_k-v_k)\, dx\\
    &+\frac{1}{2}\int_{\Omega}(\overline{v}_k-v_k)(\mathbb{C}_1-\mathbb{C}_0)\widehat{\nabla}{u^{\delta}_{D,k}(v_k)}: \widehat{\nabla}{u^{\delta}_{D,k}(v_k)}\, dx\\
    &-\frac{1}{2}\int_{\Omega}(\overline{v}_k-v_k)(\mathbb{C}_1-\mathbb{C}_0)\widehat{\nabla}{u^{\delta}_{N,k}(v_k)}: \widehat{\nabla}{u^{\delta}_{N,k}(v_k)}\, dx
\end{aligned}
\end{equation*}
Analogously to the last part of the proof of Theorem \ref{th:5.2}, we get that on the right-hand side of the previous inequality is non-zero only the first term that converges to $\|\nabla v\|_{L^2(\Omega)}$, hence $\|\nabla v_k\|_{L^2(\Omega)}\to \|\nabla v\|_{L^2(\Omega)}$, that is the assertion. 
\end{proof}

\section{The algorithm and numerical examples}\label{sec:algorithm}
For the reconstruction procedure, we adopt the method utilized in \cite{DecEllSty16}, based on a parabolic inequality and the implementation of the Primal Dual Active Set (PDAS) method.\\
For every $\delta, \varepsilon>0$, consider $v$ solution to the following parabolic inequality
\begin{eqnarray}
    && \int_\Omega \partial_t v(\omega-v) + J'_{\delta,\varepsilon}(v)[\omega-v]\geq 0, \quad \forall \omega\in\mathcal{K},\ \  t\in(0+\infty),\nonumber\\
    && v(\cdot,0)=v_0\in\mathcal{K}.\nonumber
\end{eqnarray}
Let us denote with $v^n_h\approx v(\cdot,t^n)$ and with $v^0_h= v_0\in \mathcal{K}_h$. We consider the discretized version of the parabolic inequality, using a semi-implicit time discretization, that is: given $v^0_h\in\mathcal{K}_h$ find $v^{n+1}_h\in\mathcal{K}_h$ satisfying 
\begin{equation}\label{eq:discr_parab_ineq}
\begin{aligned}
     \frac{1}{\tau_n}\int_{\Omega}
(v_h^{n+1}-v_h^n)(\omega_h-v_h^{n+1}) &-\frac{1}{2}\int_{\Omega}(\omega_h-v^{n+1}_h)(\mathbb{C}_1-\mathbb{C}_0) \widehat{\nabla}{u^{\delta}_{N,h}(v^n_h)}:\widehat{\nabla}{u^{\delta}_{N,h}(v^n_h)}\, dx \\
&+\frac{1}{2}\int_{\Omega}(\omega_h-v^{n+1}_h)(\mathbb{C}_1-\mathbb{C}_0) \widehat{\nabla}{u^{\delta}_{D,h}(v^n_h)}:\widehat{\nabla}{u^{\delta}_{D,h}(v^n_h)}\, dx\\
&+2 \gamma \varepsilon \int_{\Omega}\nabla v^{n+1}_h \cdot \nabla (\omega_h-v^{n+1}_h)\, dx\\
&+ \frac{\gamma}{\varepsilon}\int_{\Omega}(1-2v^n_h)(\omega_h-v^{n+1}_h)\, dx\geq 0, \quad \forall \omega_h\in \mathcal{K}_h, n\geq 0,
\end{aligned}
\end{equation}
where $\tau_n$ is the time step, and $u^{\delta}_{N,h}(v^n_h)\in \mathcal{V}_{h,\Sigma_D}$, and $u^{\delta}_{D,h}(v^n_h)\in \mathcal{V}_{h}$ are the discrete solutions of \eqref{pb:neumann_incl} and \eqref{pb:dirichlet_incl}, respectively, for $v=v_h^n$.
\subsection{Convergence analysis}
We now show the following result related to a monotonicity property of the algorithm based on the discrete parabolic inequality.
\begin{lemma}\label{lemma:monotonicity}
For each $n \in \mathbb{N}$, there exists a constant $c_1>0$ such that, if $\tau_n\leq (1+c_1)^{-1}$, then 
\begin{equation}
    \|v_h^{n+1}-v_h^n\|^2_{L^2(\Omega)}+J_{\delta,\varepsilon,h}(v_h^{n+1})\leq J_{\delta,\varepsilon,h}(v_h^n), 
\end{equation}
where 
$c_1=c_1(\Omega,h,\delta, \xi_0,r_0,L_0,\|\mathbb{C}_0\|_{L^{\infty}(\Omega)},\|\mathbb{C}_1\|_{L^{\infty}(\Omega)},\|u^{\delta}_N\|_{W^{1,\infty}(\Omega)},\|u^{\delta}_D\|_{W^{1,\infty}(\Omega)})$.
\end{lemma}
\begin{proof}
Let us choose $\omega_h=v^n_h$ in \eqref{eq:discr_parab_ineq}. Then,
\begin{equation*}
\begin{aligned}
    0\leq -\frac{1}{\tau_n}\|v^{n+1}_h-v^n_h\|^2_{L^2(\Omega)}&-\frac{1}{2}\int_{\Omega}(v^n_h-v^{n+1}_h)(\mathbb{C}_1-\mathbb{C}_0)\widehat{\nabla}{u^{\delta}_{N,h}(v^n_h)}:\widehat{\nabla}{u^{\delta}_{N,h}(v^n_h)}\, dx \\
    &+\frac{1}{2}\int_{\Omega}(v^n_h-v^{n+1}_h)(\mathbb{C}_1-\mathbb{C}_0)\widehat{\nabla}{u^{\delta}_{D,h}(v^n_h)}:\widehat{\nabla}{u^{\delta}_{D,h}(v^n_h)}\, dx\\
    &+2\gamma\varepsilon\int_{\Omega}\nabla v^{n+1}_h:\nabla(v^n_h-v^{n+1}_h)\, dx\\
    &+\frac{\gamma}{\varepsilon}\int_{\Omega}(1-2v^n_h)(v^n_h-v^{n+1}_h)\, dx.
\end{aligned}    
\end{equation*}
After lengthy but simple calculations, we get
\begin{equation}\label{eq:aux10}
\begin{aligned}
    &\frac{1}{\tau_n}\|v^{n+1}_h-v^n_h\|^2_{L^2(\Omega)}+\gamma\varepsilon\|\nabla(v^n_h-v^{n+1}_h)\|^2_{L^2(\Omega)}+\frac{\gamma}{\varepsilon}\|v^n_h-v^{n+1}_h\|^2_{L^2(\Omega)}\\
    &+\gamma\int_{\Omega}\left[\varepsilon|\nabla v^{n+1}_h|^2-\frac{1}{\varepsilon}v^{n+1}_h(1-v^{n+1}_h)  \right]\, dx - \gamma\int_{\Omega}\left[\varepsilon|\nabla v^n_h|^2+\frac{1}{\varepsilon}v^n_h(1-v^n_h)\right]\, dx\\
    &\leq -\frac{1}{2}\int_{\Omega}(v^n_h-v^{n+1}_h)(\mathbb{C}_1-\mathbb{C}_0) \widehat{\nabla}{u^{\delta}_{N,h}(v^n_h)}:\widehat{\nabla}{u^{\delta}_{N,h}(v^n_h)}\, dx\\
    &+\frac{1}{2}\int_{\Omega}(v^n_h-v^{n+1}_h)(\mathbb{C}_1-\mathbb{C}_0) \widehat{\nabla}{u^{\delta}_{D,h}(v^n_h)}:\widehat{\nabla}{u^{\delta}_{D,h}(v^n_h)}\, dx:=I_N+I_D.
\end{aligned}
\end{equation}
We now work on the two terms $I_N$ and $I_D$. Note that
\begin{equation*}
\begin{aligned}
    I_N&=\frac{1}{2}\int_{\Omega}(v^n_h-v^{n+1}_h)(\mathbb{C}_1-\mathbb{C}_0) \widehat{\nabla}{u^{\delta}_{N,h}(v^n_h)}:\widehat{\nabla}{u^{\delta}_{N,h}(v^n_h)}\, dx\\
    &-\int_{\Omega}(v^n_h-v^{n+1}_h)(\mathbb{C}_1-\mathbb{C}_0) \widehat{\nabla}{u^{\delta}_{N,h}(v^n_h)}:\widehat{\nabla}{u^{\delta}_{N,h}(v^n_h)}\, dx,
\end{aligned}
\end{equation*}
hence, using the discretized version of \eqref{eq:aux4}, we get
\begin{equation*}
\begin{aligned}
    I_N&=\frac{1}{2}\int_{\Omega}\left(\mathbb{C}^{\delta}(v^n_h)-\mathbb{C}^{\delta}(v^{n+1}_h)\right)\widehat{\nabla}{u^{\delta}_{N,h}(v^n_h)}: \widehat{\nabla}{u^{\delta}_{N,h}(v^n_h)}\, dx\\ &-\int_{\Omega}\mathbb{C}^{\delta}(v^n_h)\widehat{\nabla}{\widetilde{u}^{\delta}_{N,h}(v^n_h)}:\widehat{\nabla}{u^{\delta}_{N,h}(v^n_h)}\, dx
\end{aligned}
\end{equation*}
Using a discretized version of \eqref{eq:aux9} in the previous equation, we get that
\begin{equation}\label{eq:aux13}
\begin{aligned}
    I_N&=-\frac{1}{2}\int_{\Omega}\left(\mathbb{C}^{\delta}(v^{n+1}_h)-\mathbb{C}^{\delta}(v^n_h)\right)\widehat{\nabla}{u^{\delta}_{N,h}(v^n_h)}: \widehat{\nabla}{u^{\delta}_{N,h}(v^n_h)}\, dx\\ &-\int_{\Omega}\mathbb{C}^{\delta}(v^{n+1}_h)\widehat{\nabla}{(u^{\delta}_{N,h}(v^{n+1}_h)-u^{\delta}_{N,h}(v^n_h))}:\widehat{\nabla}{u^{\delta}_{N,h}(v^n_h)}\, dx\\
    &=-J^{\delta}_N(v^{n+1}_h)+J^{\delta}_N(v^{n}_h)\\
    &+\frac{1}{2}\int_{\Omega}\mathbb{C}^{\delta}(v^{n+1}_h)\widehat{\nabla}{(u^{\delta}_{N,h}(v^{n+1}_h)-u^{\delta}_{N,h}(v^n_h))}:\widehat{\nabla}{(u^{\delta}_{N,h}(v^{n+1}_h)-u^{\delta}_{N,h}(v^n_h))}\, dx
\end{aligned}
\end{equation}
Completely analogous calculations can be made for $I_D$, using the discretized versions of \eqref{eq:aux5} and \eqref{eq:aux11}. Then, by means of \eqref{eq:aux13} (and the analogous expression for $I_D$) in \eqref{eq:aux10}, we get
\begin{equation}\label{eq:aux12}
\begin{aligned}
    &\frac{1}{\tau_n}\|v^{n+1}_h-v^n_h\|^2_{L^2(\Omega)}+\gamma\varepsilon\|\nabla(v^n_h-v^{n+1}_h)\|^2_{L^2(\Omega)}+\frac{\gamma}{\varepsilon}\|v^n_h-v^{n+1}_h\|^2_{L^2(\Omega)}\\
    &+\gamma\int_{\Omega}\left[\varepsilon|\nabla v^{n+1}_h|^2-\frac{1}{\varepsilon}v^{n+1}_h(1-v^{n+1}_h)  \right]\, dx - \gamma\int_{\Omega}\left[\varepsilon|\nabla v^n_h|^2+\frac{1}{\varepsilon}v^n_h(1-v^n_h)\right]\, dx\\
    &\leq -J^{\delta}_N(v^{n+1}_h)+J^{\delta}_N(v^{n}_h)-J^{\delta}_D(v^{n+1}_h)+J^{\delta}_D(v^{n}_h)\\
    &+\frac{1}{2}\int_{\Omega}\mathbb{C}^{\delta}(v^{n+1}_h)\widehat{\nabla}{(u^{\delta}_{N,h}(v^{n+1}_h)-u^{\delta}_{N,h}(v^n_h))}:\widehat{\nabla}{(u^{\delta}_{N,h}(v^{n+1}_h)-u^{\delta}_{N,h}(v^n_h))}\, dx\\
    &+\frac{1}{2}\int_{\Omega}\mathbb{C}^{\delta}(v^{n+1}_h)\widehat{\nabla}{(u^{\delta}_{D,h}(v^{n+1}_h)-u^{\delta}_{D,h}(v^n_h))}:\widehat{\nabla}{(u^{\delta}_{D,h}(v^{n+1}_h)-u^{\delta}_{D,h}(v^n_h))}\, dx.
\end{aligned}
\end{equation}
Finally, adding and subtracting $\overline{J}_{ND}$, which is defined in \eqref{def:JND}, in \eqref{eq:aux12} we get 
\begin{equation*}
\begin{aligned}
    &\frac{1}{\tau_n}\|v^{n+1}_h-v^n_h\|^2_{L^2(\Omega)}+\gamma\varepsilon\|\nabla(v^n_h-v^{n+1}_h)\|^2_{L^2(\Omega)}+\frac{\gamma}{\varepsilon}\|v^n_h-v^{n+1}_h\|^2_{L^2(\Omega)}+ J_{\delta,\varepsilon}(v^{n+1}_h)\\
    &\leq J_{\delta,\varepsilon}(v^{n}_h)+\frac{1}{2}\int_{\Omega}\mathbb{C}^{\delta}(v^{n+1}_h)\widehat{\nabla}{(u^{\delta}_{N,h}(v^{n+1}_h)-u^{\delta}_{N,h}(v^n_h))}:\widehat{\nabla}{(u^{\delta}_{N,h}(v^{n+1}_h)-u^{\delta}_{N,h}(v^n_h))}\, dx\\
    &+\frac{1}{2}\int_{\Omega}\mathbb{C}^{\delta}(v^{n+1}_h)\widehat{\nabla}{(u^{\delta}_{D,h}(v^{n+1}_h)-u^{\delta}_{D,h}(v^n_h))}:\widehat{\nabla}{(u^{\delta}_{D,h}(v^{n+1}_h)-u^{\delta}_{D,h}(v^n_h))}\, dx.
\end{aligned}
\end{equation*}
Estimating the last two terms on the right-hand side of the previous expression, using the $H^1-$norm of the differences $u^{\delta}_{N,h}(v^{n+1}_h)-u^{\delta}_{N,h}(v^n_h)$ and $u^{\delta}_{D,h}(v^{n+1}_h)-u^{\delta}_{D,h}(v^n_h)$ in terms of $\|v^{n+1}_h-v^n_h\|_{L^{\infty}(\Omega)}$, see \eqref{estdiff} and \eqref{eq:aux15}, we have that there exists a constant
\begin{equation}\label{eq:cost c1}
c_1=c_1(\Omega,h,\delta, \xi_0,r_0,L_0,\|\mathbb{C}_0\|_{L^{\infty}(\Omega)},\|\mathbb{C}_1\|_{L^{\infty}(\Omega)},\|u^{\delta}_N\|_{H^1(\Omega)},\|u^{\delta}_D\|_{H^1(\Omega)}),
\end{equation}
such that
\begin{equation*}
\begin{aligned}
    &\left(\frac{1}{\tau_n}-c_1\right)\|v^{n+1}_h-v^n_h\|^2_{L^2(\Omega)}+\gamma\varepsilon\|\nabla(v^n_h-v^{n+1}_h)\|^2_{L^2(\Omega)}+\frac{\gamma}{\varepsilon}\|v^n_h-v^{n+1}_h\|^2_{L^2(\Omega)}+ J_{\delta,\varepsilon}(v^{n+1}_h)\\
    &\leq J_{\delta,\varepsilon}(v^{n}_h),
\end{aligned}
\end{equation*}
that is
\begin{equation*}
    \left(\frac{1}{\tau_n}-c_1\right)\|v^{n+1}_h-v^n_h\|^2_{L^2(\Omega)}+ J_{\delta,\varepsilon}(v^{n+1}_h)\leq J_{\delta,\varepsilon}(v^{n}_h).
\end{equation*}
Therefore, the assertion of the theorem follows by choosing $\tau_n\leq \frac{1}{1+c_1}$.
\end{proof}
Finally, we state a convergence result for the algorithm.
\begin{theorem}
Let $v_h^0\in\mathcal{K}_h$ be an initial guess. Under the assumptions of Lemma \ref{lemma:monotonicity}, there exists a sequence of timesteps $\tau_n$ such that $0< \beta \leq \tau_n \leq (1+c_1)^{-1} $, $\forall n>0$, where $\beta$ depends on the data and possibly on $h$. The corresponding sequence $v_h^n$ generated by \eqref{eq:discr_parab_ineq} has a convergent subsequence (still denoted by $v_h^n$) in $W^{1,\infty}$ such that 
$$v_h^n\to v_h, \qquad \textrm{as}\ n\to+\infty, $$
where $v_h\in\mathcal{K}_h$ and satisfies the discrete optimality condition 
$$J'_{\delta,\varepsilon,h}(v_h)[\omega_h-v_h]\geq 0,\quad \forall \omega_h\in \mathcal{K}_h .$$
\end{theorem}
\begin{proof}
Let us take a collection of timesteps bounded by $(1+c_1)^{-1}$, for all $n>0$. By means of Lemma \ref{lemma:monotonicity}, we have 
\begin{eqnarray}
&&\sum_{n=0}^{+\infty} \|v_h^n-v_h^{n+1}\|^2_{L^2(\Omega)}
\leq J_{\delta,\varepsilon,h}(v_h^0),\label{aux:conv:1}\\
&&\sup_{n\in\mathbb{N}} J_{\delta,\varepsilon,h}(v_h^n)\leq J_{\delta,\varepsilon,h}(v_h^0).\label{aux:conv:2}
\end{eqnarray}
Therefore, we deduce that $v_h^n$ is bounded in $W^{1,\infty}$, since in finite-dimensional spaces all the norms are equivalent, and 
\begin{equation}\label{aux:conv:3}
    \lim_{n\to +\infty} \|v_h^n-v_h^{n+1}\|^2_{L^2(\Omega)} =0.
\end{equation} 
Using the weak formulations of the forward problems for $u^{\delta}_{N,h}(v^n_h)$, and \\ $w^{\delta}_{D,h}(v^n_h)=u^{\delta}_{D,h}(v^n_h)-u^f_h$, we deduce, applying analogous arguments described in the previous sections, that $u^{\delta}_{N,h}(v^n_h)$ and $u^{\delta}_{D,h}(v^n_h)$ are bounded in $H^1(\Omega)$, hence in $W^{1,\infty}(\Omega)$, where the constants appearing in the estimates do not depend on $n$. 
This implies that, recalling \eqref{eq:cost c1}, there exists a constant $C>0$, independent on $n$, such that $c_1\leq C$, and equivalently there exists a positive constant $\beta>0$, independent of $n$, such that $\beta\leq (1+c_1)^{-1}$. Moreover, from the convergence in $W^{1,\infty}(\Omega)$, we find that there exists a subsequence of 
$(v_h^n,u^{\delta}_{N,h}(v_h^n),u^{\delta}_{D,h}(v_h^n))$ (still denoted the same) such that, as $n\to+\infty$,
$$ (v_h^n,u^{\delta}_{N,h}(v_h^n),u^{\delta}_{D,h}(v_h^n)) \to (v_h,u^{\delta}_{N,h}(v_h),u^{\delta}_{D,h}(v_h))\quad\text{in~}W^{1,\infty}(\Omega),$$
hence, 
$$u^{\delta}_{N,h}(v_h^n)\to u^{\delta}_{N,h}(v_h),\quad \text{a.e.~in~}\Omega, \qquad u^{\delta}_{D,h}(v_h^n)\to u^{\delta}_{D,h}(v_h),\quad \text{a.e.~in~}\Omega. $$
Therefore, $u^{\delta}_{N,h}(v_h)$ and $u^{\delta}_{D,h}(v_h)$ are the solutions of the discrete forward problems.
To conclude, from \eqref{eq:discr_parab_ineq} and the fact that $\tau_n \geq \beta$, we get 
\begin{eqnarray}
&&-\frac{1}{2}\int_{\Omega}(\omega_h-v^{n+1}_h)(\mathbb{C}_1-\mathbb{C}_0) \widehat{\nabla}{u^{\delta}_{N,h}(v^n_h)}:\widehat{\nabla}{u^{\delta}_{N,h}(v^n_h)}\, dx\nonumber\\
&&+\frac{1}{2}\int_{\Omega}(\omega_h-v^{n+1}_h)(\mathbb{C}_1-\mathbb{C}_0) \widehat{\nabla}{u^{\delta}_{D,h}(v^n_h)}:\widehat{\nabla}{u^{\delta}_{D,h}(v^n_h)}\, dx\nonumber\\
&&+  2 \gamma \varepsilon \int_{\Omega}\widehat\nabla v^{n+1}_h \cdot \widehat\nabla (\omega_h-v^{n+1}_h)+ \frac{\gamma}{\varepsilon}\int_{\Omega}(1-2v^n_h)(\omega_h-v^{n+1}_h)\nonumber \\ 
&&\geq -\frac{1}{\beta} \|v_h^{n+1}-v_h^n\|_{L^2(\Omega)}
\| \omega_h -v_h^{n+1}\|_{L^2(\Omega)}.\nonumber
\end{eqnarray}
Finally, using \eqref{aux:conv:3} it follows that $v_h$ satisfies the discrete optimality condition \eqref{eq:discrete_opt_cond}.
\end{proof}

\subsection{Numerical Experiments}\label{sec:numerical examples}
This section is devoted to present numerical reconstructions of cavities from an implementation of the so-called Primal Dual Active Set (PDAS) method to the variational inequality \eqref{eq:discr_parab_ineq}. PDAS has been introduced in
\cite{HintItoKun03} and it has been shown its effectiveness and robustness in the reconstruction procedures, for examples in \cite{ABCHDRR20,CRBHRA19,DecEllSty16,GLNS21,HePen19}. 
In the inverse problem context, it has been applied for the reconstruction of conductivity inclusions in \cite{DecEllSty16} and in \cite{BerRatVer18} in the case of a linear and of a semilinear elliptic equation, respectively. Recently, it has been applied for detection of elastic cavities and inclusions in \cite{AspBerCavRocVer22}. The reconstruction procedure in all previous papers is based on the use of a boundary quadratic misfit functional, not on a Kohn-Vogelius functional.\\
The aim of this section is to show that choosing $\delta$ and $\varepsilon$ sufficiently small, we are able to reconstruct elastic cavities (inclusions) of different shapes. 
Precisely, we adopt the following reconstruction algorithm. 
\begin{algorithm}[H]
\caption{Discrete Parabolic Obstacle Problem}
\label{al:algorithm}
\begin{algorithmic}
\State{Given $\textrm{tol}>0$, set $n = 0$ and $v_h^0 = v_0$, the initial guess\; }
\While{$\|{v^n_h-v^{n-1}_h}\|>\textrm{tol}$}
\State{solve the forward problem \eqref{pb:neumann_incl} with $v = v_h^n$;}
\State{solve the forward problem \eqref{pb:dirichlet_incl} with $v = v_h^n$;}
\State{determine $v^{n+1}_{{h}}$ solving \eqref{eq:discr_parab_ineq} via PDAS algorithm ;}
\State{update $n = n+1$;}
\EndWhile
\end{algorithmic}
\end{algorithm}
We focus the attention only on numerical experiments in $d=2$, performing the reconstruction procedure in a square, i.e, $\Omega=[-1,1]^2$, using a triangulation $\mathcal{T}_h$ of $\Omega$, and synthetic data generated by the Finite Element (FE) method implemented in FreeFEM++ (\cite{Hecht12}). Here, we provide some information on the implementation of the algorithm and the resolution of the foward problems.\\
\underline{\textit{Tessellation of $\Omega$:}} Given $g\in L^2(\Omega)$, the boundary measurements $f$, appearing in \eqref{eq:dirichlet_problem}, are obtained by solving the Neumann problem \eqref{eq:neumann_problem}. In order to not commit an inverse crime, which can happen solving the direct and inverse problems using the same tessellation $\mathcal{T}_h$, we use a more refined triangulation $\mathcal{T}^{ref}_h$ than $\mathcal{T}_h$ for solving the forward problem \eqref{eq:neumann_problem}. Note that $\mathcal{T}^{ref}_h$ is a tessellation of the square with cavities (holes), see Figure \ref{fig:mesh_cavity}, while $\mathcal{T}_h$ is a full tessellation of $\Omega$, see Figure \ref{fig:mesh_inverse}.
\begin{figure}[h!]
    \centering
    \begin{subfigure}{0.3\textwidth}
    \centering
    \includegraphics[width=\textwidth]{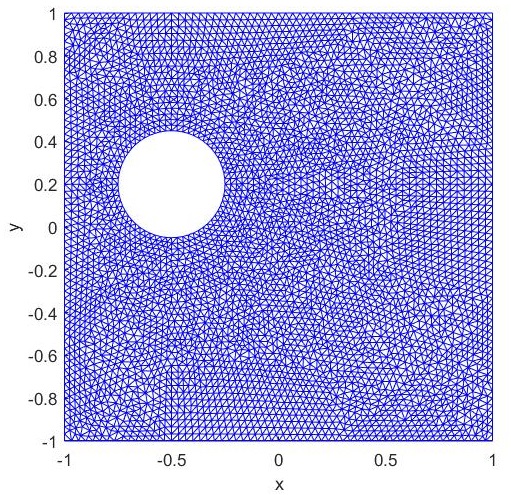}
    \caption{Mesh $\mathcal{T}^{ref}_h$ for forward problem.}
    \label{fig:mesh_cavity}
    \end{subfigure}
    \hfill
    \begin{subfigure}{0.3\textwidth}
    \centering
    \includegraphics[width=\textwidth]{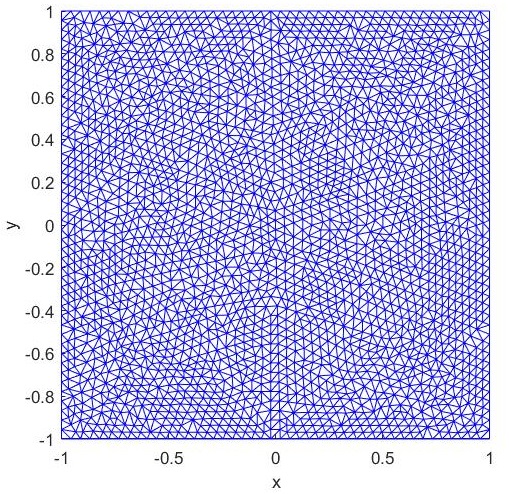}
    \caption{Mesh $\mathcal{T}_h$ for inverse problem.}
    \label{fig:mesh_inverse}
    \end{subfigure}
    \hfill
    \begin{subfigure}{0.3\textwidth}
    \centering
    \includegraphics[width=\textwidth]{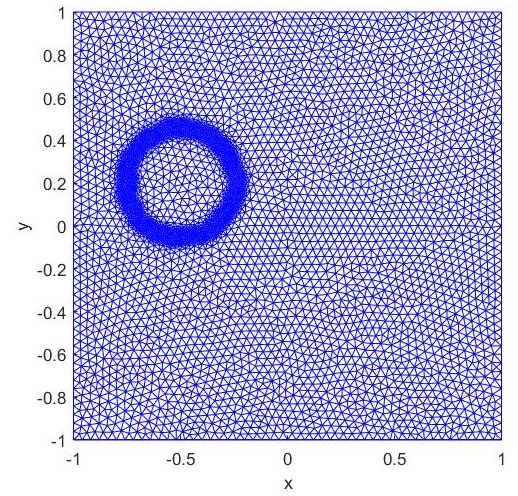}
    \caption{Refinement of the mesh around the inclusion.}
    \label{fig:refiniment}
    \end{subfigure}
\caption{Example of the meshes and the refinement.}
\label{fig:meshes}    
\end{figure}
Finally, once extracting the values of the solution of the forward problem on the boundary of the domain $\Omega$, computed by the mesh $\mathcal{T}^{ref}_h$, we interpolate these values on the mesh $\mathcal{T}_h$. In this way there is no chance to commit an inverse crime.\\
\underline{\textit{Refinement of the mesh.}} The triangular mesh $\mathcal{T}_h$ is adaptively refined during the reconstruction procedure with respect to the gradient of the phase-field variable $v_h$, see Figure \ref{fig:refiniment}. Specifically, we fix an a-priori bound and an a-priori number of iterations, which we denote by $tol_{ref}$ (with $tol_{ref}>tol$) and $n_{ref}$, respectively, such that if $\|{v^n_h-v^{n-1}_h}\|> tol_{ref}$ there is no refinement of the mesh. If $\|{v^n_h-v^{n-1}_h}\|\leq tol_{ref}$, then the refinement can occur if the remainder of $n/n_{ref}$ is equal to zero. In numerical examples, we always choose $tol_{ref}=7\times 10^{-5}$, while $n_{ref}$ is almost always $2000$ or $3000$, depending on the numerical experiment.\\  
\underline{\textit{Boundary data:}} We assume the knowledge of two different boundary measurements, that is of two pairs $(g_1,f_1)$ and $(g_2,f_2)$, where $g_1$ and $g_2$ are the given Neumann boundary conditions in \eqref{eq:neumann_problem}, while $f_1$ and $f_2$ are the measured displacement on the boundary. It is a common assumption the use of $N_m$ different boundary measurements $(g_i,f_i)$, for $i=1,\ldots,N_m$, in order to improve the numerical results. In this way, the functional to be minimized is the following one which is a  slight modification of the original optimization problem \eqref{minrel}, 
\begin{equation}\label{eq:min_prob_more_meas}
\begin{aligned}
\min_{v \in \mathcal{K}} J^{sum}_{\delta,\varepsilon}(v),& \,\,\, \\ &\hspace{-2.1cm} J^{sum}_{\delta,\varepsilon}(v) := {\frac{1}{N_m}}\sum_{i=1}^{N_m} J^{\delta}_{KV,i}(v)
+ \gamma \!\int_{\Omega}\Big( \varepsilon|\nabla v|^2 + \frac{1}{\varepsilon}v(1-v)\Big),
\end{aligned}
\end{equation}
where $J^{\delta}_{KV,i}$ is the Kohn-Vogelius functional, introduced in \eqref{def:JN and JD delta}, related to the data $(g_i,f_i)$, for $i=1,\cdots,N_m$.
The necessary optimality condition related to  \eqref{eq:min_prob_more_meas} can be equivalently obtained reasoning similarly as we did to derive \eqref{VI}.\\
In the numerical experiments, we choose $g_1=(x,y)$ and $g_2=(-y,-x)$. \\
\underline{\textit{Noise in the data:}} {Since $f_i$, for $i=1,\ldots,N_m$ are measured data, it is natural to assume that the available data are noisy perturbations of them. Therefore,} we add a uniform noise to the boundary data. Specifically, given noiseless boundary measurements $f_i\in H^{1/2}(\Sigma_N)$, for $i=1,\cdots,N_m$, the noisy data $f^{noise}_i$ is obtained by
\begin{equation*}
    f^{noise}_i=f_i+\eta\|f_i\|_{L^2(\Sigma_N)},
\end{equation*}
where $\eta$ is a random real number, $\eta\in (-a,a)$ with $a>0$, where $a$ is chosen according to the noise level. We use the following relative error to determine the noise level
\begin{equation*}
    \frac{\sqrt{\sum_{i=1}^{N_m}\|f^{noise}_i-f_i\|_{L^2(\Sigma_N)}}}{\sqrt{\sum_{i=1}^{N_m}\|f_i\|_{L^2(\Sigma_N)}}}.
\end{equation*}
\underline{\textit{Initial guess:}} In all the experiments, we assume that $v_0\equiv 0$, which corresponds to not having a-priori information on the cavity to be reconstructed. \\
Finally, we report here a table containing some of the values and ranges of the parameters utilized in most numerical tests. Possible changes in these values are highlighted in the caption of the figures related to each specific experiment.  
\begin{table}[h!]
{\begin{center}
 \begin{tabular}{|c | c | c | c | c|} 
 \hline
 \centering
 tol & $\gamma$ & $\tau_n$ & $\varepsilon$ & $\delta$  \\ [0.5ex] 
 \hline\hline
 $10^{-5}$ & $[10^{-2}, 10^{-1}]$ & $[10^{-4},10^{-3}]$ & $\frac{1}{16\pi}$  & $10^{-2}$ \\ [1ex]
 \hline
 \end{tabular}
 \captionof{table}{Values of some parameters utilized in Algorithm \ref{al:algorithm}.}\label{tab}
\end{center}
}
\end{table}

\subsubsection*{Numerical results.} In Figure \ref{fig:Test1}, we start showing the numerical experiment related to the identification of a circular inclusion in presence of noiseless measurements. One can observe the reconstruction at different time steps. 
\begin{figure}[h!]
    \centering
    \begin{subfigure}{0.3\textwidth}
    \centering
    \includegraphics[width=\textwidth]{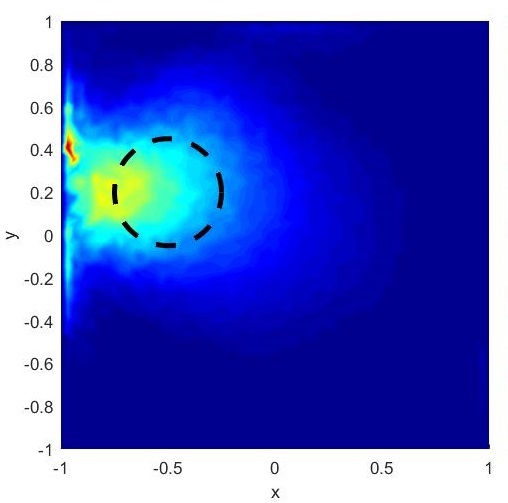}
    \caption{At $n=10$}
    \label{fig:test1_10}
    \end{subfigure}
    \hfill
    \centering
    \begin{subfigure}{0.3\textwidth}
    \centering
    \includegraphics[width=\textwidth]{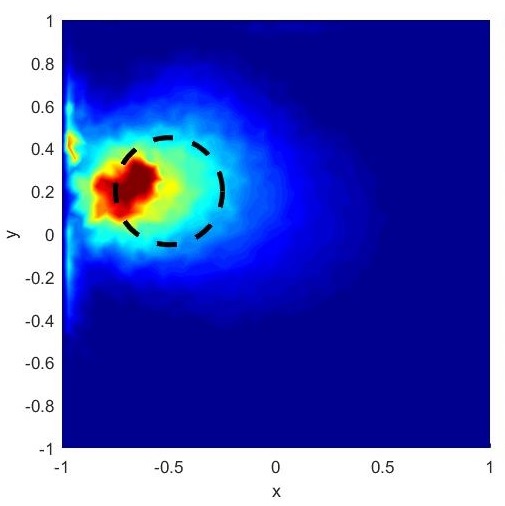}
    \caption{At $n=100$}
    \label{fig:test1_100}
    \end{subfigure}
    \hfill
    \begin{subfigure}{0.3\textwidth}
    \centering
    \includegraphics[width=\textwidth]{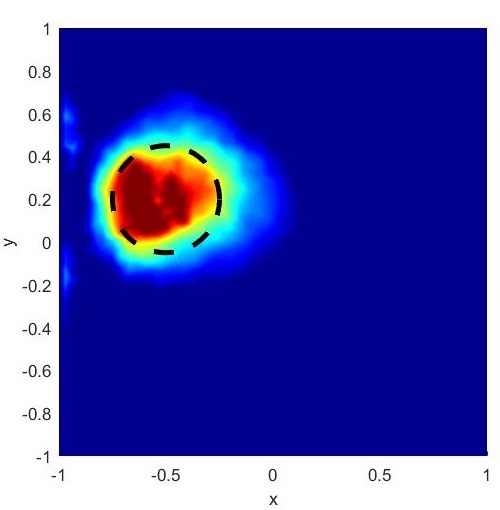}
    \caption{At $n=500$}
    \label{fig:test1_500}
    \end{subfigure}
    \hfill\\
    \begin{subfigure}{0.3\textwidth}
    \centering
    \includegraphics[width=\textwidth]{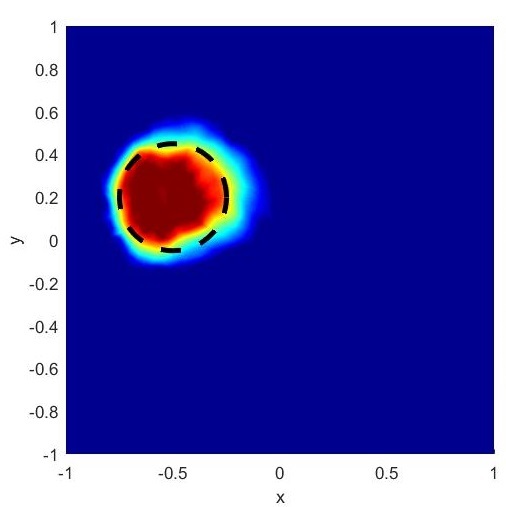}
    \caption{At $n=1000$}
    \label{fig:test1_1000}
    \end{subfigure}
     \hfill
    \begin{subfigure}{0.3\textwidth}
    \centering
    \includegraphics[width=\textwidth]{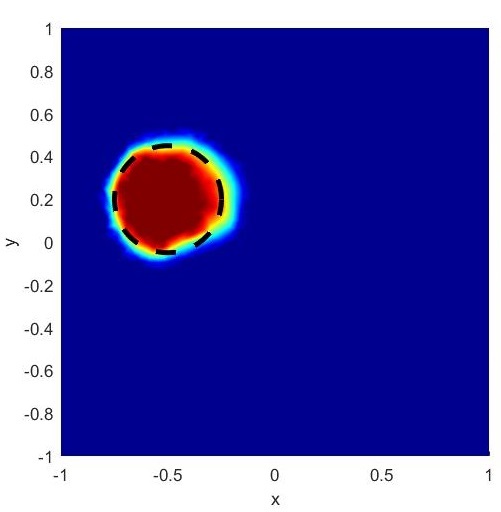}
    \caption{At $n=1500$}
    \label{fig:test1_1500}
    \end{subfigure}
    \hfill
    \begin{subfigure}{0.35\textwidth}
    \centering
    \includegraphics[width=\textwidth]{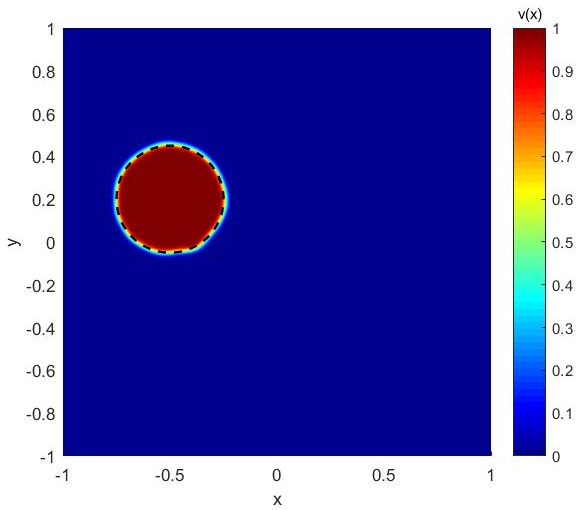}
    \caption{At $n=3033$ (final step)}
    \label{fig:test1_final}
    \end{subfigure}
\caption{Test 1. Reconstruction of a circular cavity without noise in the measurements. We provide the reconstruction at different time steps $n$. Dotted line represents the target cavity. In this test we use $n_{ref}=800$, $\tau_n=2\times 10^{-3}$, and $(\mu,\lambda)=(0.5,1)$.}
\label{fig:Test1}    
\end{figure}
In Figure \ref{fig:Test2}, we provide the same numerical example of Test 1 (Figure \ref{fig:Test1}) but considering noisy measurements, with different levels of noise.
\begin{figure}[h!]
    \centering
    \begin{subfigure}{0.3\textwidth}
    \centering
    \includegraphics[width=\textwidth]{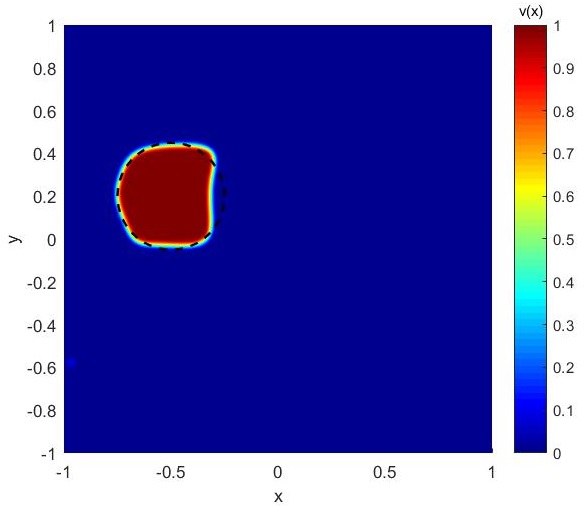}
    \caption{Noise $2\%$. Final iteration at $n=12928$. $n_{ref}=2000$, $\gamma=10^{-1}$, and $\tau_n=4\times 10^{-4}$. $(\mu,\lambda)=(0.5,1)$.}
    \label{fig:test2a}
    \end{subfigure}
    \hfill
    \centering
    \begin{subfigure}{0.3\textwidth}
    \centering
    \includegraphics[width=\textwidth]{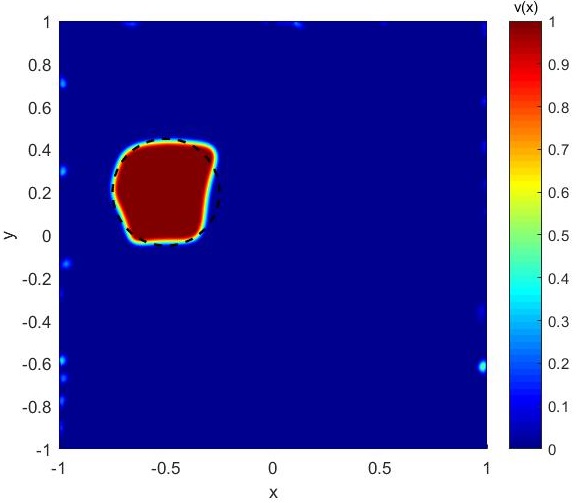}
    \caption{Noise $5\%$. Final iteration at $n=13695$. $n_{ref}=2000$, $\gamma=10^{-1}$, and $\tau_n=4\times 10^{-4}$. $(\mu,\lambda)=(0.5,1)$.}
    \label{fig:test2b}
    \end{subfigure}
    \hfill
    \begin{subfigure}{0.3\textwidth}
    \centering
    \includegraphics[width=\textwidth]{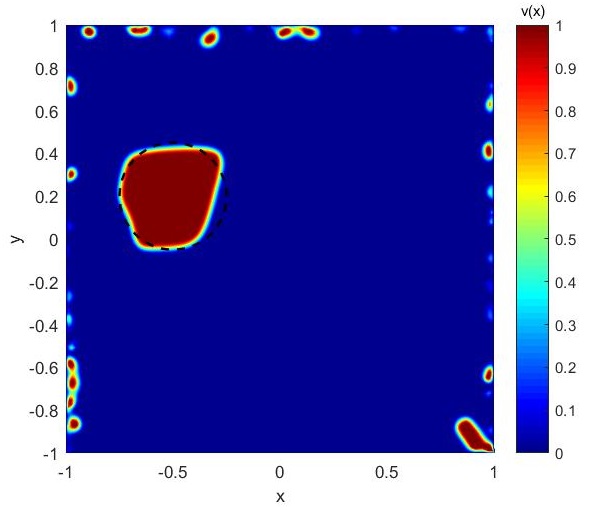}
    \caption{Noise $6,5\%$. Final iteration at $n=20978$. $n_{ref}=2000$, $\gamma=10^{-1}$, and $\tau_n=4\times 10^{-4}$. $(\mu,\lambda)=(0.5,1)$.}
    \label{fig:test2c}
    \end{subfigure}
\caption{Test 2. Reconstruction of a circular cavity with noise in the measurements. Dotted line represents the target cavity.}
\label{fig:Test2}    
\end{figure}
In Figure \ref{fig:Test3} we show the reconstruction of a circular inclusion varying the values of the Lam\'e parameters. The level of noise in this case is fixed at $5\%$. 
\begin{figure}[h!]
    \centering
    \begin{subfigure}{0.45\textwidth}
    \centering
    \includegraphics[width=\textwidth]{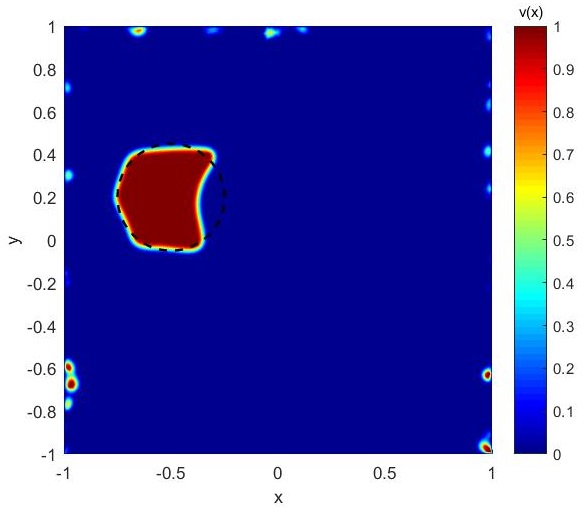}
    \caption{Noise $5\%$. Final iteration at $n=17723$. $n_{ref}=2000$, $\gamma=5\times 10^{-2}$, and $\tau_n=4\times 10^{-4}$. $(\mu,\lambda)=(1,0.2)$.}
    \label{fig:test3a}
    \end{subfigure}
    \hfill
    \centering
    \begin{subfigure}{0.45\textwidth}
    \centering
    \includegraphics[width=\textwidth]{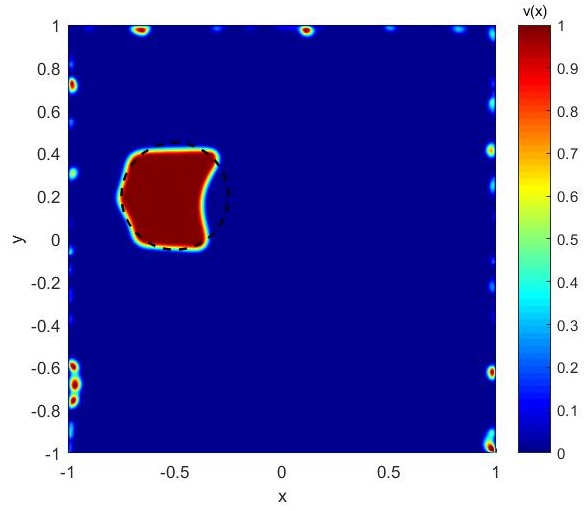}
    \caption{Noise $5\%$. Final iteration at $n=21967$. $n_{ref}=2000$, $\gamma=5\times 10^{-2}$, and $\tau_n=4\times 10^{-4}$. $(\mu,\lambda)=(1,-0.2)$.}
    \label{fig:test3b}
    \end{subfigure}
    \hfill
    \begin{subfigure}{0.45\textwidth}
    \centering
    \includegraphics[width=\textwidth]{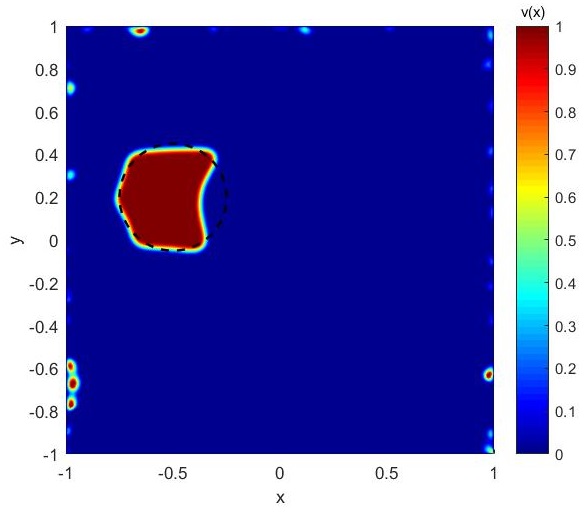}
    \caption{Noise $5\%$. Final iteration at $n=20978$. $n_{ref}=2000$, $\gamma=5\times 10^{-2}$, and $\tau_n=4\times 10^{-4}$. $(\mu,\lambda)=(0.5,0)$.}
    \label{fig:test3c}
    \end{subfigure}
    \hfill
    \begin{subfigure}{0.45\textwidth}
    \centering
    \includegraphics[width=\textwidth]{images/test3/test3c.jpg}
    \caption{Noise $5\%$. Final iteration at $n=11147$. $n_{ref}=3000$, $\gamma=10^{-1}$, and $\tau_n=10^{-4}$. $(\mu,\lambda)=(100,100)$.}
    \label{fig:test3d}
    \end{subfigure}
\caption{Test 3. Reconstruction of a circular cavity with noise in the measurements and for different values of the Lam\'e parameters.}
\label{fig:Test3}    
\end{figure}
In Figure \ref{fig:Test4}, we present the results related to the reconstruction of a rectangular cavity for different values of the noise level and $\gamma$. The Lam\'e parameters are fixed.
\begin{figure}[h!]
    \centering
    \begin{subfigure}{0.3\textwidth}
    \centering
    \includegraphics[width=\textwidth]{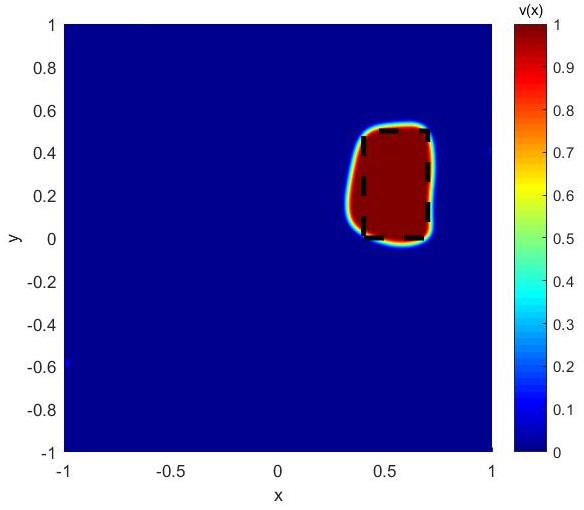}
    \caption{Noise $2\%$. Final iteration at $n=7198$. $n_{ref}=2000$, $\gamma=5\times 10^{-2}$, and $\tau_n=5\times 10^{-4}$. $(\mu,\lambda)=(0.5,1)$.}
    \label{fig:test4a}
    \end{subfigure}
    \hfill
    \centering
    \begin{subfigure}{0.3\textwidth}
    \centering
    \includegraphics[width=\textwidth]{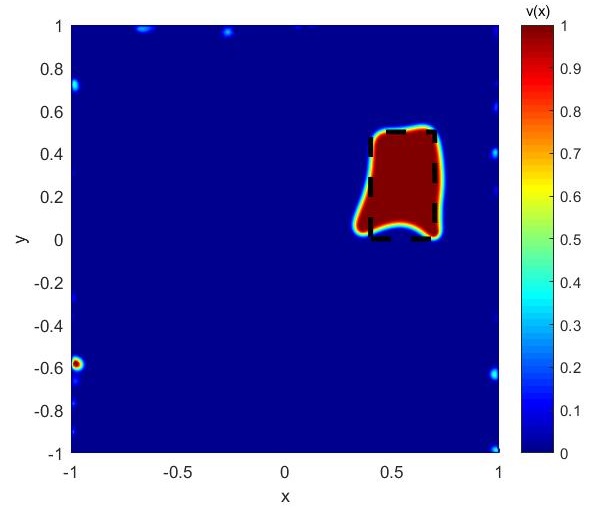}
    \caption{Noise $5\%$. Final iteration at $n=19221$. $n_{ref}=2000$, $\gamma=10^{-1}$, and $\tau_n=10^{-4}$. $(\mu,\lambda)=(0.5,1)$.}
    \label{fig:test4b}
    \end{subfigure}
    \hfill
    \begin{subfigure}{0.3\textwidth}
    \centering
    \includegraphics[width=\textwidth]{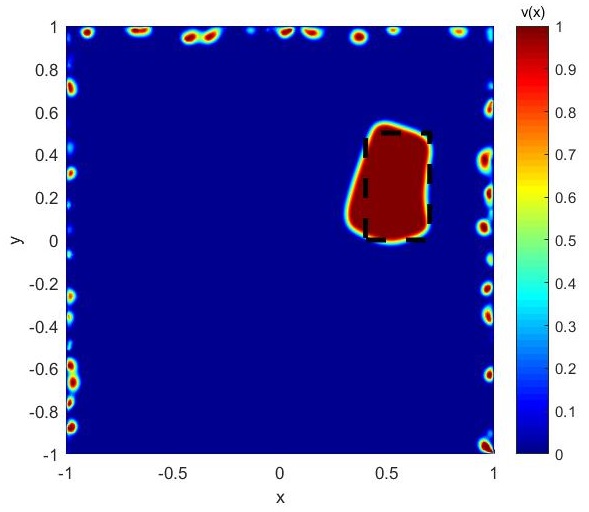}
    \caption{Noise $5\%$. Final iteration at $n=23101$. $n_{ref}=2000$, $\gamma=5\times 10^{-2}$, and $\tau_n=5\times 10^{-4}$. $(\mu,\lambda)=(0.5,1)$.}
    \label{fig:test4c}
    \end{subfigure}
\caption{Test 4. Reconstruction of a rectangular cavity with noise in the measurements. Dotted line represents the target cavity.}
\label{fig:Test4}    
\end{figure}
We also propose the case where the cavities to be reconstructed are two, see Figure \ref{fig:Test5}. We provide two examples where for the rectangular cavity we consider two different positions in $\Omega$.  
\begin{figure}[h!]
    \centering
    \begin{subfigure}{0.45\textwidth}
    \centering
    \includegraphics[width=\textwidth]{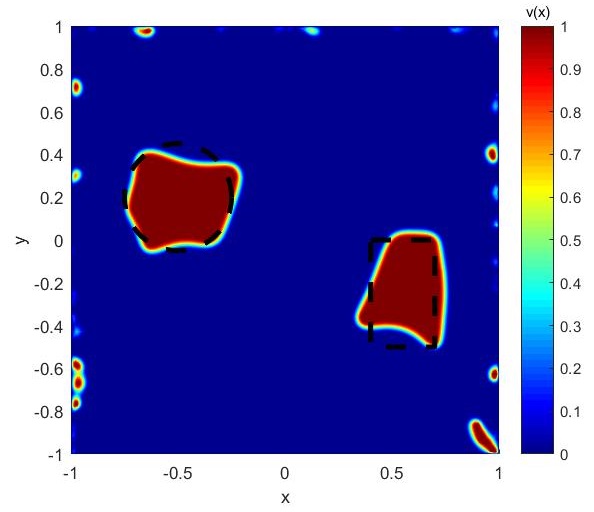}
    \caption{Noise $5\%$. Final iteration at $n=35318$. $n_{ref}=3000$, $\gamma=10^{-1}$, and $\tau_n=10^{-4}$. $(\mu,\lambda)=(0.5,1)$.}
    \label{fig:test5a}
    \end{subfigure}
    \hfill
    \centering
    \begin{subfigure}{0.45\textwidth}
    \centering
    \includegraphics[width=\textwidth]{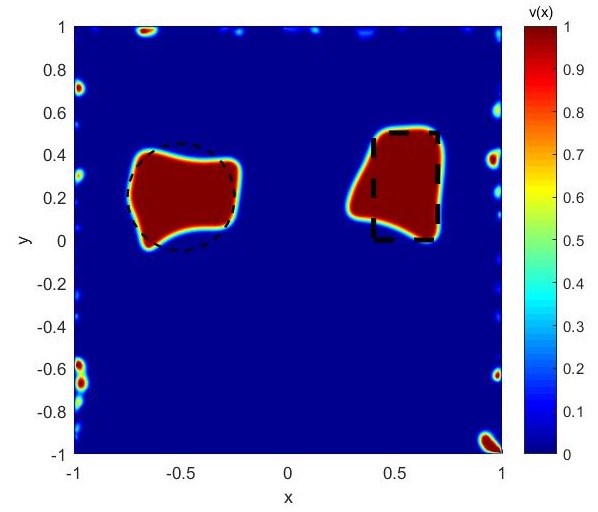}
    \caption{Noise $5\%$. Final iteration at $n=20716$. $n_{ref}=3000$, $\gamma=10^{-1}$, and $\tau_n=10^{-4}$. $(\mu,\lambda)=(0.5,1)$.}
    \label{fig:test5b}
    \end{subfigure}
\caption{Test 5. Reconstruction of two cavities with noise in the measurements. Dotted lines represent the target cavities.}
\label{fig:Test5}    
\end{figure}
In Figure \ref{fig:Test6}, we provide the numerical results of an elliptical cavity. We consider the case of noiseless measurements, the case of noise level at $2\%$ and $5\%$. Note that when the noise level is $5\%$ we change the position and the size of the cavity.
\begin{figure}[h!]
    \centering
    \begin{subfigure}{0.45\textwidth}
    \centering
    \includegraphics[width=\textwidth]{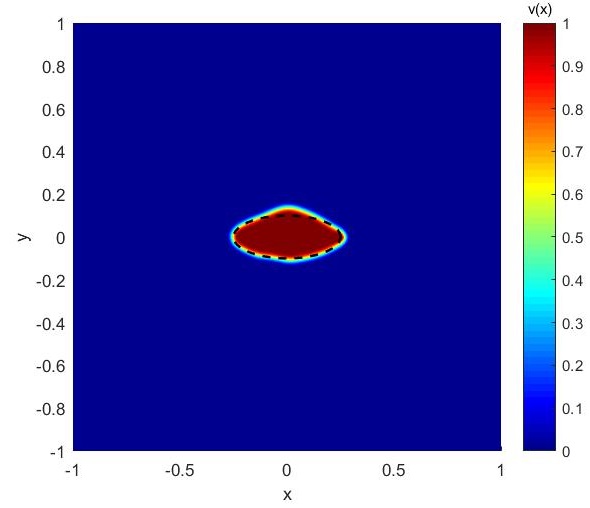}
    \caption{No noise. Final iteration at $n=5718$. $n_{ref}=2000$, $\gamma=10^{-2}$, and $\tau_n=10^{-3}$. $(\mu,\lambda)=(0.5,1)$.}
    \label{fig:test6a}
    \end{subfigure}
    \hfill
    \begin{subfigure}{0.45\textwidth}
    \centering
    \includegraphics[width=\textwidth]{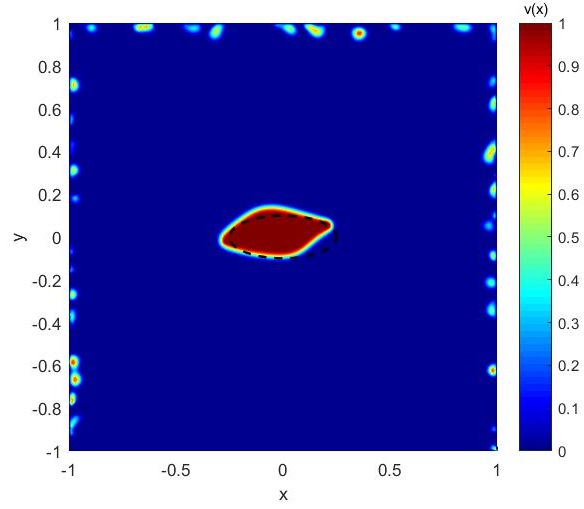}
    \caption{Noise $2\%$. Final iteration at $n=7342$. $n_{ref}=2000$, $\gamma=10^{-2}$, and $\tau_n=10^{-3}$. $(\mu,\lambda)=(0.5,1)$.}
    \label{fig:test6b}
    \end{subfigure}
    \hfill
    \begin{subfigure}{0.45\textwidth}
    \centering
    \includegraphics[width=\textwidth]{images/test6/test6b.jpg}
    \caption{Noise $2\%$. Final iteration at $n=7977$. $n_{ref}=2000$, $\gamma=5\times 10^{-2}$, and $\tau_n=10^{-3}$. $(\mu,\lambda)=(0.5,1)$.}
    \label{fig:test6c}
    \end{subfigure}
    \hfill
    \begin{subfigure}{0.45\textwidth}
    \centering
    \includegraphics[width=\textwidth]{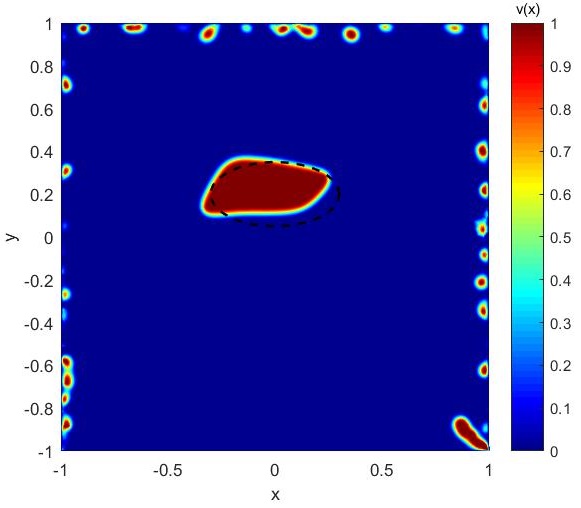}
    \caption{Noise $5\%$. Final iteration at $n=13971$. $n_{ref}=2000$, $\gamma=5\times 10^{-2}$, and $\tau_n=10^{-3}$. $(\mu,\lambda)=(0.5,1)$.}
    \label{fig:test6d}
    \end{subfigure}
\caption{Test 6. Reconstruction of an elliptical cavity with noise in the measurements.}
\label{fig:Test6}    
\end{figure}
In Figure \ref{fig:Test7} we show an example of reconstruction of a non-convex domain. We observe that the cavity is located but its non-convexity is not reconstructed. The convexification of the cavity is due to the presence of the Modica-Mortola relaxation that approximates the perimeter of the cavity. 
\begin{figure}[h!]
    \centering
    \begin{subfigure}{0.35\textwidth}
    \centering
    \includegraphics[width=\textwidth]{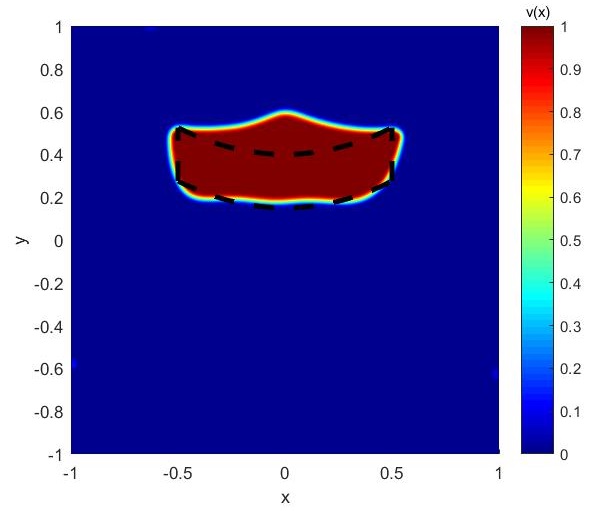}
    \caption{Noise $2\%$. Final iteration at $n=12748$. $n_{ref}=3000$, $\gamma=10^{-1}$, and $\tau_n=10^{-4}$. $(\mu,\lambda)=(0.5,1)$.}
    \label{fig:test7a}
    \end{subfigure}
    \hfill
    \begin{subfigure}{0.35\textwidth}
    \centering
    \includegraphics[width=\textwidth]{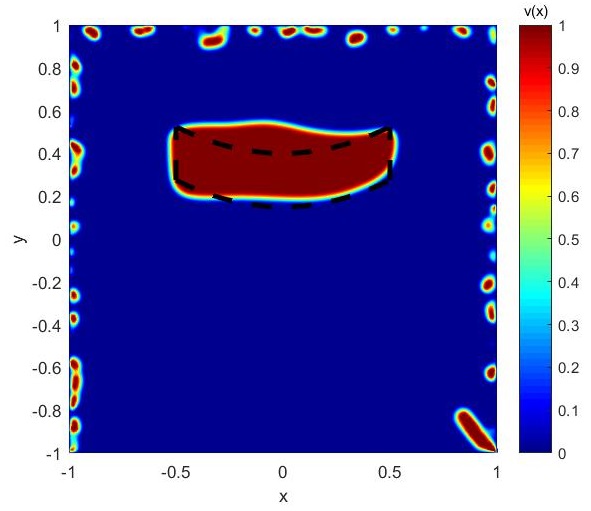}
    \caption{Noise $5\%$. Final iteration at $n=26727$. $n_{ref}=3000$, $\gamma=10^{-1}$, and $\tau_n=10^{-4}$. $(\mu,\lambda)=(0.5,1)$.}
    \label{fig:test7b}
    \end{subfigure}
\caption{Test 7. Reconstruction of a non-convex domain with noise in the measurements.}
\label{fig:Test7}    
\end{figure}
{In Figure \ref{fig:Test8}, we finally provide a numerical experiment for a comparison between the results given by $J^{sum}_{\delta,\varepsilon}$, as defined in \eqref{eq:min_prob_more_meas}, and the misfit functional studied in \cite{AspBerCavRocVer22} (see the section titled ``Numerical Examples''), which is, in the notation adopted in this paper, equal to
\begin{equation}\label{eq:misifit functional}
J^{misfit}_{\delta,\varepsilon}(v) := \frac{1}{N_m}\sum_{i=1}^{N_m}\left( \frac12 \|u^{\delta}_{N,i}(v)-f_{i}^{meas}\|_{L^2(\Sigma_N)}^2\right)
+ \gamma \!\int_{\Omega}\Big( \varepsilon|\nabla v|^2 + \frac{1}{\varepsilon}v(1-v)\Big)
\end{equation}
where $u^{\delta}_{N,i}$, for $i=1,\ldots,N_m$, are solutions to \eqref{pb:neumann_incl} with $g=g_i$. To compare the numerical outcomes of the two functionals, we use the numerical setting proposed in Figure \ref{fig:test5a}.}
\begin{figure}[h!]
    \centering
    \begin{subfigure}{0.3\textwidth}
    \centering
    \includegraphics[width=\textwidth]{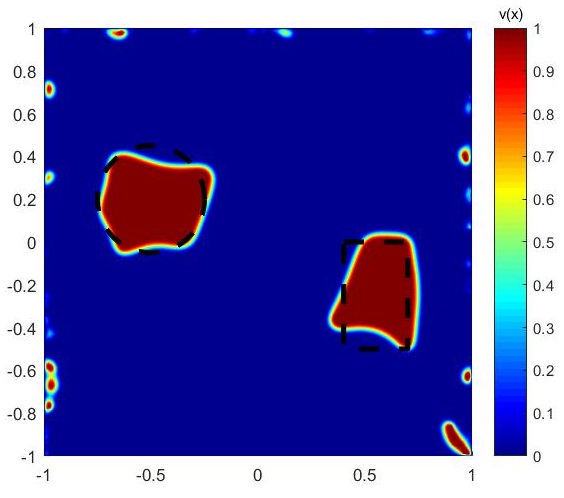}
    \caption{{Result using functional \eqref{eq:min_prob_more_meas}. Noise $5\%$. \\
    Final iteration at $n=35318$. \\ $n_{ref}=3000$, $\gamma=10^{-1}$, and $\tau_n=10^{-4}$. $(\mu,\lambda)=(0.5,1)$.}}
    \label{fig:test8a}
    \end{subfigure}
   \hfill
    \begin{subfigure}{0.3\textwidth}
    \centering
    \includegraphics[width=\textwidth]{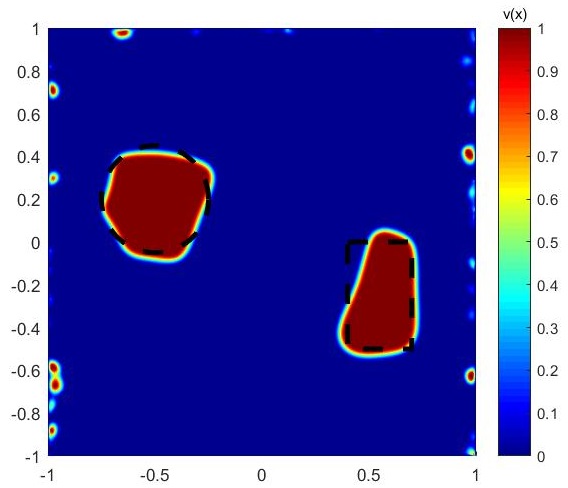}
    \caption{{Result using functional \eqref{eq:min_prob_more_meas}. Noise $5\%$. \\
    Final iteration at $n=26394$.\\ $n_{ref}=3000$,
    $\gamma=10^{-1}$, and $\tau_n=10^{-3}$. $(\mu,\lambda)=(0.5,1)$.}}
    \label{fig:test8b}
    \end{subfigure}
    \hfill \ \\
    \begin{subfigure}{0.3\textwidth}
    \centering
    \includegraphics[width=\textwidth]{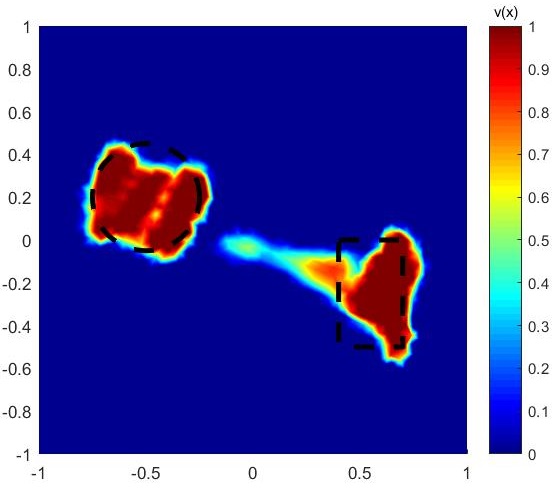}
    \caption{{Result using functional \eqref{eq:misifit functional}. Noise $5\%$. \\ 
    Final iteration at $n=28251$.\\ $n_{ref}=3000$,
    $\gamma=10^{-2}$, \\ $\tau_n=10^{-4}$. $(\mu,\lambda)=(0.5,1)$.}}
    \label{fig:test8c}
    \end{subfigure}
    \hfill
    \begin{subfigure}{0.3\textwidth}
    \centering
    \includegraphics[width=\textwidth]{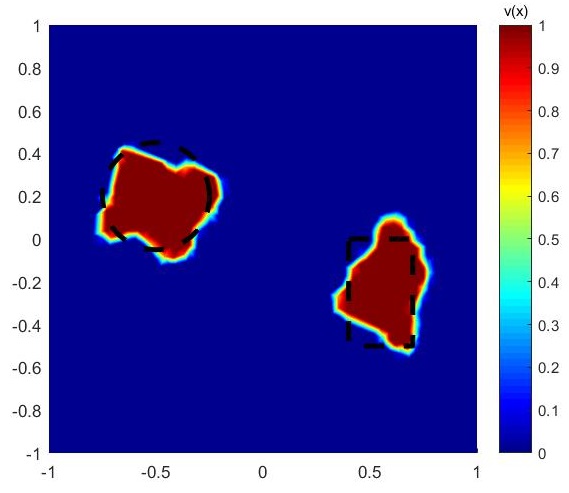}
    \caption{{Result using functional \eqref{eq:misifit functional}. Noise $5\%$. \\ 
    Final iteration at $n=15566$.\\ $n_{ref}=3000$,
    $\gamma=10^{-2}$, and $\tau_n=10^{-3}$. $(\mu,\lambda)=(0.5,1)$.}}
    \label{fig:test8e}
    \end{subfigure}
        \hfill
    \begin{subfigure}{0.3\textwidth}
    \centering
        \includegraphics[width=\textwidth]{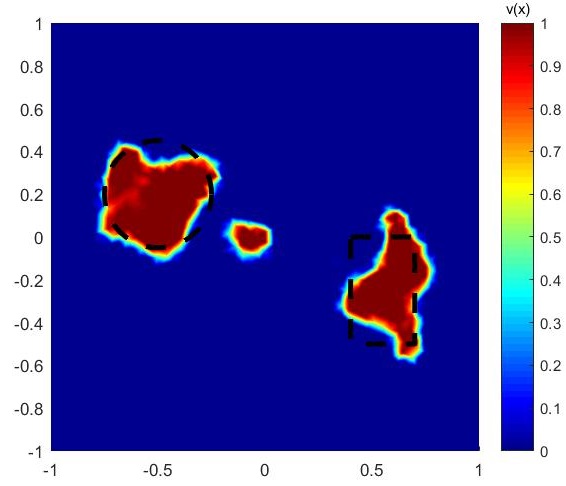}
    \caption{{Result using functional \eqref{eq:misifit functional}. Noise $5\%$. \\ 
    Final iteration at $n=26165$.\\ $n_{ref}=3000$,
    $\gamma=5\times 10^{-3}$, \\ $\tau_n=10^{-4}$. $(\mu,\lambda)=(0.5,1)$.}}
    \label{fig:test8d}
    \end{subfigure}
\caption{{Test 8. Comparison of the numerical outcomes given by the use of the functionals \eqref{eq:min_prob_more_meas} (first line in the figure) and \eqref{eq:misifit functional} (second line in the figure) tuning the values of the regularization parameter $\gamma$ and the time step $\tau_n$.}}
\label{fig:Test8}    
\end{figure}
\section{Conclusions}\label{sec:conclusions}
In this paper we have introduced a phase-field approach for a Kohn-Vogelius type functional for the reconstruction of cavities. This type of functionals is typically used in the implementation of reconstruction algorithms for the identification of defects (cavities, inclusions, cracks) embedded in a domain via shape derivative and topological derivate tools (see the introduction, Section \ref{sec:introduction}, for some literature on the topic).\\
Numerical results of our approach show a robust, efficient, {and promising} algorithm, at least in the case of convex domains. {In fact, a comparison between the misfit functional, defined in \eqref{eq:misifit functional} and studied in \cite{AspBerCavRocVer22}, and the Kohn-Vogelius type functional \eqref{eq:min_prob_more_meas} seems to show moderately better results in the case of the regularized Kohn-Vogelius type functional (see Figure \ref{fig:Test8}) in the presence of multiple inclusions. However, it should also be noted that the Kohn-Vogelius functional provides reconstructions with more artifacts around the boundary of the domain compared to the misfit functional \eqref{eq:misifit functional}. The numerical outcomes in the case of one single inclusion, such as a circle, an ellipse, or a rectangle, are equivalent for the two functionals.}  For non-convex domains it is necessary to introduce some modifications in the perimeter functional which are able to mimic the non-convexity of the domain in order to get better numerical results. From the analytical point of view, it is open the problem of proving that the minima of the relaxed functional $J_{\delta,\varepsilon}$ converge to those of the functional $J_{reg}$ through, for example, the $\Gamma$-convergence theory. 
{Moreover, in order to make the problem closer to possible applications, it would be interesting to consider, both in the analytical and the numerical framework, the case where there is an uncertainty on the knowledge of the material property, introducing, for example, some noise in the Lam\'e parameters.}

\section*{Acknowledgments}
The author is a member of GNAMPA (Gruppo Nazionale per l’Analisi Matematica, la Probabilità e le loro Applicazioni) of INdAM (Istituto Nazionale
di Alta Matematica). {This research has been performed in the framework of the MIUR-PRIN Grant 2020F3NCPX
``Mathematics for industry 4.0 (Math4I4)''}.
The author deeply thanks E. Beretta and E. Rocca for introducing him into the very interesting world of phase field methods.

\bibliographystyle{plain}
\bibliography{references.bib}

\end{document}